\documentclass[12pt]{article}

\usepackage{amssymb}
\usepackage{amsmath}
\usepackage{amsbsy}
\usepackage{amscd}
\usepackage{amsfonts}
\usepackage{amsthm}
\usepackage{mathrsfs}
\usepackage{verbatim}
\usepackage[colorlinks]{hyperref}
\usepackage{fullpage}
\usepackage{mathdots}
\usepackage{graphicx,subfigure}
\usepackage[english]{babel}
\usepackage[utf8]{inputenc}
\usepackage{tikz}
\usepackage{stackrel}
\usepackage{authblk}

\usepackage{tikz-cd}

\usepackage{mathtext}

\usetikzlibrary{backgrounds,fit, matrix}
\usetikzlibrary{positioning}
\usetikzlibrary{calc,through,chains}
\usetikzlibrary{arrows,shapes,snakes,automata, petri}

\usepackage{booktabs}
\usepackage{adjustbox}

\newtheorem{Theorem}[equation]{Theorem}
\newtheorem{Corollary}[equation]{Corollary}
\newtheorem{Lemma}[equation]{Lemma}
\newtheorem{Proposition}[equation]{Proposition}

\theoremstyle{definition}
\newtheorem{Definition}[equation]{Definition}
\newtheorem{Example}[equation]{Example}

\newtheorem{Remark}[equation]{Remark}

\numberwithin{equation}{section}
\numberwithin{figure}{section}

\newcommand{\C}{{\mathbb C}}
\newcommand{\Z}{{\mathbb Z}}
\newcommand{\Q}{{\mathbb Q}}

\newcommand{\N}{{\mathbb N}}

\newcommand{\mc}[1]{\mathcal{#1}}

\newcommand{\mt}[1]{\text{#1}}

\newcommand{\mbf}[1]{\mathbf{#1}} 

\newcommand{\x}{\times}

\newcommand{\conj}[2]{~^{#1}\!{#2}}

\begin{document}

\title{Sects}

\author{Aram Bingham, Mahir Bilen Can}

\maketitle

\begin{abstract}

By explicitly describing a cellular decomposition 
we determine the Borel invariant cycles that generate the 
Chow groups of the quotient of a reductive group by a Levi subgroup. 
For illustrations we consider the 
variety of polarizations $\mbf{SL}_n / \mbf{S}(\mbf{GL}_p\times \mbf{GL}_q)$, 
and we introduce the notion of a sect for describing its cellular decomposition.
In particular, for $p=q$, we show that the Bruhat order on the 
sect corresponding to the dense cell is isomorphic, as a 
poset, to the rook monoid with the Bruhat-Chevalley-Renner order.

\vspace{.5cm}


\end{abstract}

\normalsize

\section{Introduction}\label{S:1}

We start with a basic but important example for our purposes. 
Let $n$ be a positive integer, and 
let $\mbf{Fl}_n$ denote the variety of full flags in $\C^n$,
$$
\mbf{Fl}_n:= \{ F_\bullet\ :\ 0=V_0 \subset V_1 \subset V_2 \subset \cdots \subset V_n 
= \C^n,\ \dim_\C V_i = i\ \text{ for } i\in \{1,\dots, n\} \}.
$$
The special linear group $\mbf{SL}_n$ (with entries from $\C$) 
acts on $\mbf{Fl}_n$ by its defining representation.
Let $\mbf{L}_{p,q}$ denote the subgroup which consists of 
block matrices of the form 
$\begin{bmatrix} g_{11} & 0  \\ 0 & g_{22} \end{bmatrix}$
with $g_{11}\in \mbf{GL}_p$, $g_{22}\in \mbf{GL}_q$, and $\det g_{11} \det g_{22} = 1$.
Recently, there has been much progress in the study of the induced action 
\begin{align}\label{A:action}
\mbf{L}_{p,q} \times \mbf{Fl}_n \longrightarrow  \mbf{Fl}_n.
\end{align} 
For example, the generating series of the number of $\mbf{L}_{p,q}$-orbits is found; 
a bijection between the set of $\mbf{L}_{p,q}$-orbits and a certain set of labeled Delannoy paths
is constructed~\cite{CanUgurlu1}, and the maximal chains in the weak order are studied,
~\cite{CanJoyceWyser}; the covering relations of the Bruhat order on the set of 
$\mbf{L}_{p,q}$-orbit closures are found,~\cite{Wyser16}; 
Schubert-like polynomials corresponding to the cohomology classes of $\mbf{L}_{p,q}$-orbit
closures are described,~\cite{WyserYong}; singularities of 
$\mbf{L}_{p,q}$-orbit closures are analyzed,~\cite{McGovern,WooWyserYong, WooWyser}.

In this paper, by using some combinatorial 
tools that we developed in the first cited reference,
we describe a relationship between 
the cohomology classes corresponding to the $\mbf{L}_{p,q}$-orbit closures in 
$\mbf{Fl}_n$ and the cohomology classes of a Grassmann variety by using lattice paths in the plane. 
In fact, we achieve this connection by proving some general results on the generators 
of the cohomology ring of the quotient by a Levi subgroup of a complex reductive algebraic group. 
Before stating our main results let us briefly mention what is easy to 
see about the Chow rings of quotients by Levi subgroups.

Let $G$ be a reductive complex algebraic group. 
We fix a maximal torus $T$ and a Borel subgroup $B$ containing $T$ in $G$. 
Let $P$ be a parabolic subgroup in $G$ with unipotent radical $R_u(P)$. Let $L$ be 
a Levi subgroup of $P$ so that $P= L\ltimes R_u(P)$. We assume that $B\subseteq P$
and that $T\subseteq L$. Note that $R_u(P)$ is a unipotent subgroup of $P$ and 
it acts simply transitively on $P/L$.
Therefore, the fibration  
\begin{align}\label{A:fibration1}
P/L \rightarrow G/L \rightarrow G/P
\end{align}
is a fibre bundle with group $P/L_0$, 
where $L_0$ is the largest subgroup of $L$ 
which is normal in $P$,~\cite[Section 7.2]{Steenrod}.
In other words, $G/L$ is an affine bundle over the flag variety $G/P$. 
Since (\ref{A:fibration1}) is an affine bundle on $G/P$, the pullback map,
$A_i(G/P) \rightarrow A_{n+i}(G/L)$, between the $i$-th Chow groups is 
an isomorphism, see~\cite[Lemma 2.2]{Totaro}.
As a consequence, by using the groups of cycles graded by codimension, 
we see that the Chow rings $A^*(G/L)$ and $A^*(G/P)$ are isomorphic. 

On the other hand, by a result of Brion,
see~\cite[Corollary 12]{Brion98}, we know that 
there is an isomorphism between the rational cohomology rings 
and rational Chow rings,
\begin{align}\label{A:beyond}
A^*(G/L)_\Q \cong H^*(G/L,\Q) \ \text{ and } \ A^*(G/P)_\Q \cong H^*(G/P,\Q).
\end{align}
Of course, the second isomorphism is known to be true not only with rational coefficients 
but also with integer coefficients. To state our 
first main result, we need some more notation.
Let $T$ be a maximal torus of $G$ such that $T\subseteq L$. Let $W$ denote the 
Weyl group of $(G,T)$, and let $W_L$ denote the Weyl group of $(L,T)$.  
We let $S$ denote
symmetric algebra over the integers of the character group of $T$.
There is a canonical injection between the ring of invariants $S^W \hookrightarrow S^{W_L}$. 
Let $(S_+)^{W}$ denote the ideal of $S^{W_L}$ generated by the images of 
the homogeneous elements of positive degree from $S^W$. 
\vspace{.2cm}

\begin{Theorem}\label{T:Intro1}
In the notation as above, we have the isomorphisms 
$$
A^*(G/L)\cong H^*(G/L,\Z) \cong S^{W_L}/(S_+)^W.
$$
\end{Theorem}

The above theorem shows that the integral Chow ring of $G/L$ is torsion free. 
This gives a partial answer to a question of Brion about the integral Chow rings
of homogeneous spaces $G/H$, see~\cite[Question 1]{Brion98}.
We obtain a proof of Theorem~\ref{T:Intro1} by describing an explicit cell decomposition of $G/L$; 
the Poincar\'e duals of the closures of these cells generate the cohomology ring. 
In particular, this shows that $G/L$ is a linear variety in the sense of 
Jannsen~\cite[Section 14]{Jannsen}.

As we hinted before, the main example we have in mind for the application of our theorem
is the homogeneous space $\mbf{E}_{p,q} := \mbf{SL}_n/ \mbf{L}_{p,q}$,
which we call {\em the variety of (zero) polarizations}.
Let $S_n$ denote the symmetric group of permutations on $\{1,\dots, n\}$,
and let $S_p\times S_q$ denote its Young subgroup consisting of 
permutations $\sigma\in S_n$ such that $\sigma(\{1,\dots, p\})= \{1,\dots, p\}$. 
Then the integral Chow ring of $\mbf{E}_{p,q}$ is isomorphic to
\begin{align}\label{A:Coinvariants}
A^*(\mbf{E}_{p,q}) \cong (\Z[ x_1,\dots, x_n ] / I_+)^{S_p\times S_q},
\end{align}
where $I_+$ is the ideal generated by the elementary symmetric polynomials $e_i(x_1,\dots, x_n)$
$(1\leq i \leq n)$.
The right hand side of (\ref{A:Coinvariants}) is also the Chow ring of 
the Grassmannian $\mbf{Gr}(p,n) = \mbf{SL}_n/\mbf{P}_{p,q}$.
Here, $\mbf{P}_{p,q}$ is the parabolic subgroup containing $\mbf{L}_{p,q}$
as a Levi factor. 
It is well known that the closures of $\mbf{B}_n$-orbits in $\mbf{Gr}(p,n)$
give a $\Z$-basis for the Chow ring, where $\mbf{B}_n$ is the 
Borel subgroup consisting of 
invertible upper triangular matrices in $\mbf{SL}_n$. 
More generally, if a connected solvable group $B$ 
acts on a complex scheme $X$, then the ($B$-equivariant) Chow groups of $X$ 
are generated by the $B$-invariant cycles, see~\cite{FMSS},
so, we look more carefully at the $\mbf{B}_n$-invariant cycles in $\mbf{E}_{p,q}$. 

In the second part of our paper we identify the dense and closed 
$\mbf{B}_n$-orbits in the $\mbf{B}_n$-invariant cycles 
which form a $\Z$-basis for $A_*(\mbf{E}_{p,q})$. 
This analysis led us to the discovery that 
the combinatorics of the Borel subgroup orbits in $\mbf{E}_{p,q}$ is closely 
related to the Borel subgroup orbits in an algebraic monoid. 
To explain, let us denote by $\mbf{DB}_n$ the {\em doubled-Borel 
subgroup} $\mbf{B}_n^e\times \mbf{B}_n^e$, where 
$\mbf{B}_n^e$ is the Borel subgroup of all invertible upper triangular 
matrices in $\mbf{GL}_n$. Clearly, $\mbf{B}_n^e$ is generated 
by $\mbf{B}_n$ and the center of $\mbf{GL}_n$. 
Let $\mbf{Mat}_n$ denote the monoid of $n\times n$ 
matrices with entries from $\C$. 
The doubled-Borel subgroup $\mbf{DB}_n$ acts on $\mbf{Mat}_n$ by 
$(a,b)\cdot g \mapsto a g b^{-1}$, $a,b\in \mbf{B}_n^e$, $g\in \mbf{Mat}_n$. 
The orbits of this action are parametrized by the rook monoid, denoted by $R_n$, 
which consists of $n\x n$, 0/1-matrices with at most one 1 in each row and each column, 
see~\cite{Renner86}. 
The rook monoid is partially ordered with respect to {\em Bruhat-Chevalley-Renner order} 
defined by 
\begin{align}\label{A:BCR}
\sigma \leq \tau \iff \mbf{DB}_n\cdot \sigma \subseteq 
\overline{\mbf{DB}_n\cdot \tau } \qquad (\sigma,\tau \in R_n).
\end{align}

The Borel orbits in $\mbf{E}_{p,q}$ are parametrized by the combinatorial objects called
`signed $(p,q)$-involutions.' By definition, a {\em signed $(p,q)$-involution} is an involution 
$\sigma$ in $S_n$ whose fixed points are labeled either by + or by - in such a way that 
the number of +'s is $p-q$ more than the number of -'s. Here, without loss of generality, 
we assume that $p\geq q$.
The data of a signed $(p,q)$-involution can equivalently be given by another combinatorial 
object called a $(p,q)$-clan. We will describe this object in more detail in the sequel but let
us mention that the equivalence of these two notions are explained in detail in~\cite{CanUgurlu2}.  
By the work of Yamamoto~\cite{Yamamoto} and Wyser~\cite{WyserThesis},
it is known that the Borel orbits in $\mbf{E}_{p,q}$ are parametrized by the $(p,q)$-clans.
By using our `cellular decomposition', we partition the set of $(p,q)$-clans into subsets
that we call {\em sects}. These equivalence classes 
have geometric significance that we explain here.
There is a canonical projection map
from the symmetric space $\mbf{E}_{p,q}$ onto the Grassmannian, 
$\pi_{\mbf{L}_{p,q},\mbf{P}_{p,q}}: \mbf{E}_{p,q} \longrightarrow \mbf{Gr}(p,n)$.
A cellular decomposition of $\mbf{Gr}(p,n)$ is given by the orbits of $\mbf{B}_n$, 
which are called the {\em Schubert cells}.  
A Schubert cell is uniquely determined by a lattice path in a $p\x q$-grid.
We show that two Borel orbits in $\mbf{E}_{p,q}$ corresponding to 
$(p,q)$-clans $\gamma$ and $\tau$ project down onto
the same Schubert cell under $\pi_{\mbf{L}_{p,q},\mbf{P}_{p,q}}$ 
if and only if $\gamma$ and $\tau$ belong to the same sect. 
In other words, we show that the sects are parametrized by the 
lattice paths in a $p\x q$-grid. 

It turns out that each sect is a graded poset with a unique maximal and 
a unique minimal element. The poset structure is induced from the 
inclusion order on the Borel orbit closures in $\mbf{E}_{p,q}$. 
Let us denote by $\mt{Dense}(p,q)$ 
the sect containing the clan corresponding to the dense Borel orbit in $\mbf{E}_{p,q}$. 
We will call $\mt{Dense}(p,q)$ the {\em big sect}.
An upper order ideal in a poset $\mc{P}$ is a subset $\mc{I}\subseteq \mc{P}$
such that if $x$ and $y$ are two elements from $\mc{P}$ with 
$x\in \mc{I}$ and $x\leq y$, then $y\in \mc{I}$. 
Now we can state our second main result.

\begin{Theorem}\label{T:main2}
Let $p$ and $q$ be two positive integers. Then 
the big sect $\mt{Dense}(p,q)$ is a maximal upper order ideal in the 
Bruhat poset of $(p,q)$-clans. 
Furthermore, if $p=q$, then there exists a poset isomorphism between 
$\mt{Dense}(p,q)$ and the rook monoid $R_p$. 
\end{Theorem}
In Theorem~\ref{T:main2}, the partial order on $R_p$ is the one that is defined at (\ref{A:BCR}).
It is not hard to guess how the second part of this theorem should generalize for not necessarily
equal integers $p$ and $q$, 
and it is not hard to guess how it could be related to algebraic monoids of other types.
We leave these investigations to a future paper.


Now we are ready to give a brief description of our paper. 
In Section~\ref{S:2}, we setup our notation and explain the basic fiber bundle 
structure $\mbf{E}_{p,q} \rightarrow \mbf{Gr}(p,n)$. 
In the following Section~\ref{S:3}, 
we prove our main first main theorem. 
We start Section~\ref{S:4} by 
reviewing Yamamoto's parametrization~\cite{Yamamoto} of the Borel 
group orbits in the symmetric space $\mbf{E}_{p,q}$, and we review 
various interpretations of clans. 
The purpose of Section~\ref{S:5} is to review the work of Wyser on 
Bruhat order on clans. In the final Section~\ref{S:6} we prove 
our second main result.

\vspace{1cm}

\textbf{Acknowledgements.} The second author is partially supported
by a grant from Louisiana Board of Regents.

\section{Preliminaries}\label{S:2}

The purpose of this section is to setup some notation 
and to present some basic results pertaining to the homogeneous spaces
that we are concerned about. 

Throughout our paper, all varieties and groups are defined over $\C$.
Unless otherwise stated,  
$G$ will always stand for a reductive algebraic group, and 
$B$ will always stand for a Borel subgroup in $G$. 
A closed subgroup of $G$ is called parabolic subgroup if it contains a Borel subgroup. 
A parabolic subgroup $P$ is called standard (with respect to $B$) 
if the inclusion $B\subseteq P$ holds.
We use the notation $R_u( H )$ to denote the unipotent radical of 
an algebraic group $H$. 
A subgroup $L$ in $G$ is called a Levi subgroup 
if for some parabolic subgroup $P$ the following 
decomposition holds: $P= L \ltimes R_u(P)$.

\subsection{The fiber bundle structure on the variety of polarizations.}\label{SS:1}

Let $\mbf{B}_n$ denote as before the Borel subgroup of upper triangular 
matrices in $\mbf{SL}_n$. 
For two positive integers $p$ and $q$ such that $n=p+q$, we denote
by $\mbf{Mat}_{p,q}$ the affine space of $p\times q$ matrices with entries
from $\C$. 
The parabolic subgroup that is defined by 
$$
\mbf{P}_{p,q} := \left\{
\begin{bmatrix}
x & h \\
0 & y
\end{bmatrix} :\ (x,y)\in \mbf{GL}_p\times \mbf{GL}_q,\ \det x\det y =1,\ h\in \mbf{Mat}_{p,q} \right\}
$$
contains the subgroup 
$$
\mbf{L}_{p,q} := \left\{
\begin{bmatrix}
x & 0 \\
0 & y
\end{bmatrix} :\ (x,y)\in \mbf{GL}_p\times \mbf{GL}_q,\ \det x\det y =1 \right\}
$$
as a Levi subgroup.

We identify $\mbf{SL}_n/\mbf{B}_n$ with the flag variety $\mbf{Fl}_n$, 
and we denote the set of orbits of $\mbf{L}_{p,q}$ in $\mbf{Fl}_n$ 
by $\mbf{L}_{p,q}\backslash \mbf{SL}_n/\mbf{B}_n$. Thus, the elements of 
the quotient space $\mbf{L}_{p,q} \backslash \mbf{Fl}_n$ 
are in canonical bijection with the $(\mbf{L}_{p,q},\mbf{B}_n)$ double cosets in $\mbf{SL}_n$. 
It is well known that $\mbf{L}_{p,q}$ is a `spherical subgroup', that is to say,  
$\mbf{L}_{p,q}\backslash \mbf{SL}_n/\mbf{B}_n$ is a finite set.

As we promised in the introduction, 
we will now compute $\mbf{L}_0:=\bigcap_{g\in \mbf{P}_{p,q}} g \mbf{L}_{p,q} g^{-1}$. 
First of all, it is easy to observe that  
\begin{align}\label{A:easy}
\mbf{L}_0 = \bigcap_{g\in R_u(\mbf{P}_{p,q})} g \mbf{L}_{p,q} g^{-1}.
\end{align}
Indeed, (\ref{A:easy}) is a simple consequence of the fact that 
$\mbf{P}_{p,q}= R_u(\mbf{P}_{p,q}) \rtimes \mbf{L}_{p,q}$. 
Now, let $g$ be an element from $R_u(\mbf{P}_{p,q})$ so it has the form 
$$
g= 
\begin{bmatrix}
id_p & h \\
0 & id_q
\end{bmatrix} 
$$
where $id_p$ and $id_q$ are identity matrices of sizes $p$ and $q$, respectively, 
and $h\in \mbf{Mat}_{p,q}$. 
Let $a$ denote 
$$
a:= 
\begin{bmatrix}
x & 0 \\
0 & y
\end{bmatrix} \in \mbf{L}_{p,q}
$$
Then 
$$
g a g^{-1}= 
\begin{bmatrix}
id_p & h \\
0 & id_q
\end{bmatrix} 
\begin{bmatrix}
x & 0 \\
0 & y
\end{bmatrix}
\begin{bmatrix}
id_p & -h \\
0 & id_q
\end{bmatrix} 
=
\begin{bmatrix}
x & -xh + h y \\
0 & y
\end{bmatrix}
$$
Such a matrix is contained in $\mbf{L}_{p,q}$ if and only if 
$-xh + h y=0$ for all $h\in \mbf{Mat}_{p,q}$,
which holds true if and only if 
$x=c\, id_p$ and $y=c\, id_q$ where $c\in \C^*$ so that 
$c\, id_n$ is contained in the center of $\mbf{SL}_n$.
Therefore, $\mbf{L}_0$ is 
equal to the center of $\mbf{SL}_n$, which is the finite cyclic 
group of order $n$, denoted by $\mu_n$. 
Therefore, the original fibre bundle 
$\mbf{SL}_n/\mbf{L}_{p,q}\rightarrow \mbf{SL}_n/\mbf{P}_{p,q}$
is associated to the 
principal fibre bundle $\mbf{PSL}_n \rightarrow \mbf{SL}_n/\mbf{P}_{p,q}$ with 
group $\mbf{P}_{p,q}/\mu_n$ in the sense of~\cite{Steenrod}.

\subsection{The Bruhat-Chevalley decomposition.}

We will loosely follow the notation from Borel's book~\cite{Borel}.

Let $T$ be a maximal torus of $G$ contained in $B$. 
Then $B = T \ltimes U$, where 
$U$ is the unipotent radical of $B$. 
We denote the Weyl group of $(G,T)$ by the letter $W$.
Then $W= N_G(T)/T$, where $N_G(T)$ is the normalizer of $T$ in $G$. 
We will allow ourselves to confuse the elements of $W$ with their
representing elements in $N_G(T)$.

The Borel subgroup that is opposite to $B$ is denoted by $B^-$. 
Note that $B\cap B^- = T$ is the common Levi subgroup of $B$ and $B^-$. 
The unipotent radical of $B^-$ is denoted by $U^-$.

{\em Notation: 
If $g$ is an element from $G$ and $H$ is a subgroup of $G$,
then we denote by $\conj{g}{H}$ the subgroup  
$\conj{g}{H}:= \{ g h g^{-1}: h\in H\}$.}

Now let $w\in W$ and consider the following two subgroups of $U$:
\begin{align}\label{A:two subgroups}
U_w := U\cap \conj{w}{U} \qquad \text{ and } \qquad U_w' := U \cap \conj{w}{(U^-)}.
\end{align}
Also, we need some root theoretic notation; let $\varPhi$ denote 
the root system associated with $(G,T)$ and 
let $\varPhi^+$ denote the set of positive roots determined by $B$.
For $\alpha\in \varPhi^+$, the notation $U_\alpha$ 
stands for the 1-dimensional $T$-stable unipotent subgroup of $U$ so that we have
\begin{align}
\mt{Lie}(U) = \bigoplus_{\alpha \in \varPhi^+} \mt{Lie}(U_\alpha).
\end{align}
\\

{\em Caution: $U_\alpha$ should not to be confused with $U_w$ of (\ref{A:two subgroups}).} 
\\

The two subgroups $U_w$ and $U_w'$ in (\ref{A:two subgroups}) of $U$ are 
$T$-stable, therefore, they are {\it directly spanned} by the root subgroups 
$U_\alpha$ ($\alpha >0$) that they contain. 
The sets of such $\alpha$'s are given, respectively, by 
\begin{align}
\varPhi^+ (U_w) &= \{ \alpha > 0 :\ \alpha^w >0\} \\
\varPhi^+ (U_w') &= \{ \alpha > 0 :\ \alpha^w < 0\} 
\end{align}
where $\alpha^w$ is the root $\alpha \circ Int(n)$ for any $n\in N_G(T)$
representing $w$. 
\begin{Remark}
For two elements $u,v\in W$, we have
$\varPhi^+ (U_u) = \varPhi^+ (U_v)$ if and only if $u=v$.
\end{Remark}
Since $\varPhi^+(U_w)$
and $\varPhi^+(U_w')$ partition $\varPhi^+$, we know that
$$
U = U_w \cdot U_w'= U_w' \cdot U_w.
$$

\begin{Theorem}[Bruhat-Chevalley decomposition of $G$]
$G$ is the disjoint union of the double cosets $BwB$
$(w\in W)$. 
If $w\in W$ then the morphism 
\begin{align*}
U_w' \times B &\longrightarrow BwB \\
(x,y) &\longmapsto xwy
\end{align*}
is an isomorphism of varieties.
In particular, $Uw B = BwB$.
\end{Theorem}

\begin{proof}
See the proof of \cite[Theorem 14.12]{Borel}.
\end{proof}

If $w\in W$, then we denote by $C(w)$ the $(B,B)$-double coset 
\begin{align}
C(w) : = B w B
\end{align}
in $G$.
The parabolic version of the Bruhat decomposition goes as follows.
Let $J$ be a subset of the set of simple roots $\varDelta$ 
that is determined by $\varPhi^+$ (hence by $B$). 
Let $P_J$ denote the standard parabolic subgroup that is 
determined by $J$. Thus, $P_J$ is generated by $B$ 
and the set of representatives in $N_G(T)$ of the 
simple reflections $s_\alpha$, where $\alpha \in J$.

Let $W_J$ denote the parabolic subgroup of $W$ 
generated by $s_\alpha$, where $\alpha \in J$.
Let $W^J$ denote the set of minimal length left-coset
representatives of $W_J$ in $W$. In this case, 
for each $w\in W$, there exist unique $w_J\in W_J$ 
and a unique $w^J \in W^J$ such that 
$$
w = w^J w_J.
$$
Finally, let us denote the double coset $BwB$ in $G$ by $C(w)$.

\begin{Theorem}\label{T:ParabolicBruhat}
We preserve the notation from the previous paragraph. 
In addition, let $\pi$ (resp. $\pi_J$) denote the canonical projections
$\pi: G\rightarrow G/B$ (resp. $\pi_J:G\rightarrow G/P_J$).
\begin{enumerate}
\item We have $Bw P_J = B w^J P_J$ and $\pi_J (C(w)) = \pi_J (C(w^J))$.
\item If $\sigma_J$ denotes the canonical surjection $\sigma_J: G/B \rightarrow G/P_J$,
then the induced map $\sigma_J : \pi(C(w)) \rightarrow \pi_J(C(w^J))$ is surjective. 
It is injective if and only if $w\in W^J$. In this case, $\sigma_J :  \pi(C(w)) \rightarrow \pi_J(C(w^J))$
is an isomorphism of varieties. 
\item The sets $\pi_J(C(w))$ ($w\in W^J$) form a partition of the variety $G/P_J$.
\end{enumerate}
\end{Theorem}

\begin{proof}
See~\cite[Proposition 21.29]{Borel}.
\end{proof}

\subsection{A description of the Bruhat order on involutions.}

In this final subsection of our preliminaries section, we briefly review
the descriptions of the Bruhat order on involutions, partial involutions,
and the rook matrices. 

Let $n$ be a positive integer, and let $\mbf{DB}_n$ denote the Borel subgroup 
$\mbf{B}_n\times \mbf{B}_n$ of the doubled general linear group $\mbf{GL}_n\times \mbf{GL}_n$. 
As we mentioned before, the $\mbf{DB}_n$-orbits in $\mbf{Mat}_n$ 
are parametrized by the rook monoid, $R_n$, which is partially ordered 
with respect to Bruhat-Chevalley-Renner order defined as (\ref{A:BCR}). 
There are various combinatorial ways of re-writing this partial order, 
see~\cite{PPR,MillerSturmfels}. We present the description that is 
given in~\cite{MillerSturmfels}, where Miller and Sturmfels consider
$\mbf{B}^-_n\x \mbf{B}_n$-orbit closures in $\mbf{Mat}_n$.
The resulting Bruhat poset on $R_n$ is isomorphic to the one 
that is described in~\cite{PPR}.

Let $M$ be an element from $R_n$. 
For indices $i,j\in \{1,\dots, n\}$, the {\em $(i,j)$-th upper-left submatrix}
of $M$ is the $i \x j$ submatrix $M_{i\x j}$ whose $(r,s)$-th entry
is the $(r,s)$-th entry of $M$ for $r\in\{1,\dots, i\}$ and $s\in \{1,\dots,j\}$.
Let us denote the rank of a matrix $A$ by $\mt{rank}(A)$. 
The {\em rank-control matrix} of $M$ is the $n\times n$ 
matrix $\mt{Rk}(M)$ whose $(i,j)$-th entry is the rank of $M_{i\x j}$
for $i,j\in \{1,\dots, n\}$. In other words, $\mt{Rk}(M)=(\mt{rank}(M_{i\x j}))_{i,j=1}^n$.
\begin{Example}
Let 
\begin{align}\label{A:example1}
M= 
\begin{bmatrix}
1 & 0 & 0 & 0 & 0 & 0 \\
0 & 0 & 0 & 0 & 0 & 0 \\
0 & 0 & 0 & 0 & 1 & 0 \\
0 & 1 & 0 & 0 & 0 & 0 \\
0 & 0 & 1 & 0 & 0 & 0 \\
0 & 0 & 0 & 0 & 0 & 0 
\end{bmatrix}
\implies 
\mt{Rk}(M) = 
\begin{bmatrix}
1 & 1 & 1 & 1 & 1 & 1 \\
1 & 1 & 1 & 1 & 1 & 1 \\
1 & 1 & 1 & 1 & 2 & 2 \\
1 & 2 & 2 & 2 & 3 & 3 \\
1 & 2 & 3 & 3 & 4 & 4 \\
1 & 2 & 3 & 3 & 4 & 4 
\end{bmatrix}
\end{align}
\end{Example}
In this terminology, the Bruhat order on $R_n$ is given by 
$M\leq M'$ if and only if $\mt{rank}(M_{i\x j})\geq \mt{rank}(M'_{i\x j})$
for all $i,j\in \{1,\dots, n\}$, see~\cite[Lemma 15.19]{MillerSturmfels}.
In Section~\ref{S:6}, we will employ rank control matrices 
which are defined in terms of \emph{lower-left submatrices}, 
with the analogous Bruhat order.

We now consider (partial) involutions. 
The Bruhat order on involutions is studied by 
Richardson and Springer in the references~\cite{RS90,RS94},
and more combinatorially by Incitti in~\cite{Incitti}.
The description of the Bruhat order on `partial involutions' 
in terms of rank-control matrices is due to Bagno and Cherniavsky~\cite{BagnoCherniavsky}.
Here, by a {\em partial involution} we mean a 
rook matrix $M\in R_n$ which is symmetric, that is, $M= M^\top$.
If, in addition, $\mt{rank}(M)=n$, then the partial involution is 
the matrix representation of an involution, that is, $M^2 =id$,
and $M\in S_n$. Let $P_n$ denote the set of all partial 
involutions in $R_n$, and let $I_n$ denote the set of all
involutions in $S_n$. Then $I_n \subset P_n\subset R_n$.

The {\em congruence Bruhat order} on $P_n$ (and on $I_n$)
is defined as follows:
For $M$ and $M'$ two partial involutions from $P_n$, 
$M\leq M'$ if and only if $\mbf{O}(M)\subseteq \overline{\mbf{O}(M')}$,
where $\mbf{O}(M)$ (resp. $\mbf{O}(M')$) is the 
$\mbf{B}_n$-orbit of $M$ (resp. of $M'$) under the action of 
$\mbf{B}_n$ on the variety of all symmetric $n\x n$ matrices
via $B\cdot M = (B^{-1})^\top  M B^{-1}$ ($B\in \mbf{B}_n$). 
It is shown in~\cite{BagnoCherniavsky} that the congruence 
order on $P_n$ agrees with the restriction of the Bruhat order,
in terms of rank-control matrices, from $R_n$
to $P_n$. The restriction of the congruence Bruhat order from $P_n$ to $I_n$
agrees with the Bruhat order that is studied by 
Richardson-Springer and Incitti.

\section{Proof of Theorem~\ref{T:main1}}\label{S:3}

We preserve the notation from the previous section.
Let $P_J$ be a standard parabolic subgroup and let 
$L$ be a Levi subgroup of $P_J$ such that $T\subseteq L$. 
Since $R_u(P_J)$ is normal in $P_J$, we know 
that $R_u(P_J)$ is a $T$-stable unipotent subgroup of $U$.
Also, if $U_L$ denotes the intersection $L\cap U$, that is, the 
maximal unipotent subgroup of the Borel subgroup $B_L:= L\cap B$,
then $U_L$ is $T$-stable as well. Moreover, the unipotent subgroups $R_u(P_J)$ and 
$U_L$ are complementary in the sense that  
$U = U_L \cdot R_u(P_J) = R_u(P_J) \cdot U_L$.
We proceed with reviewing some standard facts.

\begin{Lemma}\label{L:1}
Let $P$ be a parabolic subgroup and let $L$ be a Levi subgroup 
such that $P= R_u(P)\rtimes L$. 
Then the map
\begin{align*}
R_u(P) \times L &\longrightarrow P \\
(x,y) &\longmapsto xy
\end{align*}
is an isomorphism of varieties.
\end{Lemma}
\begin{proof}
The semi-direct product $H\rtimes K$ of two algebraic groups $H$ and $K$ 
is isomorphic to $H\times K$ as a variety, see~\cite[Chapter I, 1.11]{Borel}. 
\end{proof}

\begin{Lemma}\label{L:2}
Let $P^-$ denote the parabolic subgroup opposite to $P$ such that 
$P\cap P^- = L$. Then the map
\begin{align*}
R_u(P^-) \times P &\longrightarrow G \\
(x,y) &\longmapsto xy
\end{align*}
is an isomorphism onto an open subset of $G$. 
The image is equal to $P^-\cdot P$ in $G$.
\end{Lemma}
\begin{proof}
See~\cite[Proposition 14.21]{Borel}.
\end{proof}

By combining Lemmas~\ref{L:1} and~\ref{L:2} we obtain 
the following result.

\begin{Proposition}\label{P:12}
The map
\begin{align}\label{A:P1}
R_u(P^-) \times R_u(P) \times L &\longrightarrow G \\
(x,y,z) &\longmapsto xyz \notag
\end{align}
is an isomorphism onto $P^-\cdot P$, which is an open subset of $G$.
Moreover, we have
\begin{enumerate}
\item the map (\ref{A:P1}) is compatible with the projection $\pi_L :G\rightarrow G/L$. 
In other words, the map  
\begin{align*}
\phi_L: R_u(P^-) \times R_u(P) &\longrightarrow G/L \\
(x,y) &\longmapsto xyL
\end{align*}
is an isomorphism onto an open subvariety of $G/L$.
\item 
The map (\ref{A:P1}) is compatible with the projection $\pi_P :G\rightarrow G/P$. 
In other words, the map   
\begin{align}\label{A:P12}
\phi_P: R_u(P^-) &\longrightarrow G/P \\
x &\longmapsto xP
\end{align}
is an isomorphism onto an open subvariety of $G/P$.
\end{enumerate}
\end{Proposition}

\begin{Remark}\label{R:Bactiononradical}
Since $L$ normalizes $R_u(P)$, we see that $P$,
hence $B$, acts on $R_u(P)$ by conjugation. 
Moreover, since $R_u(P)\cong R_u(P^-)$, if we 
twist, that is to say conjugate, the action of $P$ on its unipotent
radical by the longest element of $W^P$, we get
an action on the unipotent radical $R_u(P^-)$.
Thus, the left multiplication action of $B$ on the open subvariety 
$\phi_P(R_u(P^-)) \subset G/P$ is obtained from a suitably 
twisted (conjugated) conjugation action of $B$ on $R_u(P)$. 
\end{Remark}

As an immediate consequence of Proposition~\ref{P:12} we 
have the following corollary.

\begin{Corollary}\label{C:P:12}
Let $\pi_P$, $\pi_L$, and $\pi_{L,P}$ denote the canonical projections
\begin{align*} 
\pi_P &: G\rightarrow G/P, \\
\pi_L  &: G\rightarrow G/L, \\
\pi_{L,P} &: G/L\rightarrow G/P.
\end{align*}
If $V$ denotes $\phi_P (R_u(P^-))$ in $G/P$, 
then $\pi_{L,P}^{-1}(V) \cong R_u(P) \times V$ as a variety.
Furthermore, the map  $\tilde{s}: V \rightarrow R_u(P)\times V$ 
defined by $\tilde{s}(uP) = (id,uP)$ induces a local section $s$ 
of $\pi_{L,P}$,
that satisfies $\pi_{L,P} \circ s (uP) = uP$ for all $uP\in V$.
\end{Corollary}

Now we are ready to describe a local trivialization of the fibre bundle
$\pi_{L,P}: G/L\rightarrow G/P$.

Let $V$ denote, as in Corollary~\ref{C:P:12}, the open set
$V= \phi_P(R_u(P^-))$ in $G/P$. This set is called the 
{\em canonical affine open neighborhood} 
of the identity coset $id P$.
To obtain the canonical affine open neighborhood of another 
point $gP$ in $G/P$, we simply translate $V$ to 
\hbox{$g\cdot V:= g \cdot R_u(P^-) P/P = gR_u(P^-) P/P$.}
\\

{\em Notation:
From now on we will denote the canonical affine open neighborhood
of a point $gP \in G/P$ by $V(g)$, so, 
\begin{align}
V(g) := g\cdot V.
\end{align}
In particular, we have $V(id) = V$.}
\\

Let $y,w$ be two elements from $W^P$, and 
let $X(w)$ denote the Schubert variety $\overline{B\cdot wP}$
in $G/P$. If $y\leq w$ in Bruhat order, then the 
{\em canonical $T$-stable transversal to $yP$ in $X(w)$} is defined by 
\begin{align}\label{A:canonicalTstabletransversal}
\mc{N}_{y,w}:= ( \conj{y}{ R_u(P^-)} \cap U^-) yP) \cap X(w).
\end{align}

\begin{Lemma}\label{L:BrionPolo}
The canonical $T$-stable transversal $\mc{N}_{y,w}$ 
is a closed, $T$-stable, subvariety of 
$$
(yR_u(P^-) idP) \cap X(w) = V(y) \cap X(w).
$$
Furthermore, there exists a $T$-equivariant isomorphism,
\begin{align}
V(y) \cap X(w) \cong (B\cdot yP) \times \mc{N}_{y,w}.
\end{align}
\end{Lemma}
\begin{proof}
See the discussion at the end of~\cite[Section 1.2]{BrionPolo}.
\end{proof}

\begin{Lemma}\label{L:finiteopencover}
The finite collection $\{ V(y) \}_{y\in W^P}$ forms an open cover of $G/P$.
\end{Lemma}

\begin{proof}
In the light of the Bruhat decomposition, 
if we consider the canonical $T$-stable transversals $\mc{N}_{y,w_0}$, then 
our claim follows from the fact that the canonical affine open
neighborhood $V(y)$ of $yP$ contains the $B$-orbit $B\cdot yP$ in $G/P$,
see Lemma~\ref{L:BrionPolo}.
\end{proof}

According to~\cite{Totaro14}, a {\em linear scheme} is a scheme which can be obtained 
by an inductive procedure starting with an affine space of any dimension, in such a way that 
the complement of a linear scheme imbedded in affine space is also a linear scheme,
and a scheme which can be stratified as a finite disjoint union of linear schemes is a linear scheme. 
Now we are ready to prove the main (new) result of this section.

\begin{Theorem}\label{T:main1}
The canonical projection $\pi_{L,P}: G/L\rightarrow G/P$ has the structure 
of an affine fibre bundle.
The variety $G/L$ is linear, in particular, $G/L$ has a finite disjoint decomposition 
into subvarieties of the form $\C^a \times (\C^*)^b$.
Moreover, the preimages under $\pi_{L,P}$ of the Borel orbits 
are affine open subvarieties of $G/L$. 
\end{Theorem}

\begin{proof}
Let $\{ V(y) \}_{y\in W^P}$ be the open cover of $G/P$ as in Lemma~\ref{L:finiteopencover}.
By the discussion following Corollary~\ref{C:P:12}, 
we know that, over $V(y)$, the projection $\pi_{L,P}$ has a local section, and we have 
\begin{align}\label{A:localdecomposition}
\pi_{L,P}^{-1}(V(y)) \cong R_u(P)\times V(y).
\end{align}
This proves the first claim.

By restricting the isomorphism (\ref{A:localdecomposition}) 
to the quasi-affine subvariety $\pi_{L,P}^{-1}(B\cdot yP)$,  
and by using Lemma~\ref{L:BrionPolo} with $w=w_0$, 
we find that
\begin{align}\label{A:localdecomposition2}
\pi_{L,P}^{-1}(B\cdot yP) \cong R_u(P) \times B\cdot yP.
\end{align}
Since $\bigcup_{y\in W^P} B\cdot yP = G/P$, we have
\begin{align}\label{A:partitioning}
\bigcup_{y\in W^P} \pi_{L,P}^{-1}(B\cdot yP) = G/L.
\end{align}
Moreover, if $v,y\in W^P$ are two elements such that $v\neq y$,
then $\pi_{L,P}^{-1}(B\cdot yP) \cap \pi_{L,P}^{-1}(B\cdot vP) = \emptyset$.
In other words, (\ref{A:partitioning}) is a partition of $G/L$.
Since each of the factors on the right hand side of (\ref{A:localdecomposition2})
is an affine space, the proof of our theorem is finished.
\end{proof}

\begin{Corollary}\label{C:T:main1}
Let $y$ be an element from $W^P$ and let $C(y,L)$ denote
the preimage $\pi_{L,P}^{-1}(X(y))$ of the Schubert variety $X(y)$ in $G/L$.
Then we have $\pi_{L,P} \vert_{C(y,L)}: C(y,L) \rightarrow X(y)$ as an
affine fibre bundle over $X(y)$, and furthermore, $C(y,L)$ is a linear subvariety of $G/L$.
\end{Corollary}
\begin{proof}
The proof is a special case of the proof of Theorem~\ref{T:main1}.
\end{proof}

\begin{Proposition}\label{C:minmax1}
We preserve our previous notation. 
Let $O$ be a $B$-orbit in $G/P$. 
Then $\pi^{-1}_{L,P}(O) \subset G/L$ contains a unique closed $B$-orbit.
In addition, if we assume that 
$L$ is a spherical Levi subgroup of $G$, then $\pi^{-1}_{L,P}(O) \subset G/L$
has a unique dense $B$-orbit.
\end{Proposition}

\begin{proof}
Since $\pi_{L,P}$ is $B$-equivariant, the preimage of $O$ in $G/L$ is $B$-stable.
Let us first assume that $L$ is a spherical subgroup. 
Since $\pi^{-1}_{L,P}(O)$ contains only finitely many $B$-orbits, 
exactly one of them is dense. This proves the second claim. 
To prove that there exists a unique closed $B$-orbit in $\pi^{-1}_{L,P}(O)$, we 
look more closely at the action on the fibers. 

Suppose that the orbit $O$ is of the form $O= B\cdot yP/P$ ($y\in W^P$). 
We already know from the proof of Theorem~\ref{T:main1} that 
$\pi_{L,P}^{-1}(B\cdot yP) \cong R_u(P)\times B\cdot yP$,
see~\ref{A:localdecomposition2}. 
The action of $B$ on the right hand side of this isomorphism is given by 
the diagonal action. The action of $B$ on the first factor in (\ref{A:localdecomposition2}), 
that is $R_u(P)$, is the twisted conjugation action,
see Remark~\ref{R:Bactiononradical}. 
Next, we will make use of the following general observation.
\begin{Remark}\label{R:general}
Let $H$ be a connected algebraic group acting on two irreducible 
varieties $Z_1$ and $Z_2$. Then a minimal dimensional $H$-orbit
in the diagonal action of $H$ on $Z_1\times Z_2$ is contained in 
$O_1\times O_2$, where $O_i$ ($i\in \{1,2\}$) is a minimal 
dimensional orbit in $Z_i$ ($i\in \{1,2\}$). 
Furthermore, if $H$ acts transitively on $Z_2$, then the minimal 
dimensional orbit of $H$ in $Z_1\times Z_2$ is of the form 
$O_1\times Z_2$, where $O_1$ is a minimal dimensional $H$-orbit in $Z_1$. 
\end{Remark}

Now we go back to our proof. 
In light of Remark~\ref{R:general}, it suffices to concentrate on 
the action of $B$ on the first factor. Furthermore, 
since we are interested in the dimension of the closed orbit with minimal dimension only, 
we ignore the twisting by $y$, so, we focus on the ordinary conjugation action of $B$ on $R_u(P)$. 
This action has a unique fixed point, that is the $B$-orbit of the identity element of $R_u(P)$.
Therefore, our proof is finished.
\end{proof}

\begin{Definition}\label{D:firstdef}
We propose to call any set of combinatorial objects which parametrize 
the $B$-orbits in $\pi^{-1}_{L,P}(O)$ the {\em sect} of $O$. 
\end{Definition}


We are ready to prove one of the main results that we announced 
in the introduction. It states that there are isomorphisms 
$$
A^*(G/L)\cong H^*(G/L,\Z) \cong S^{W_L}/(S_+)^W.
$$
\begin{proof}[Proof of Theorem~\ref{T:Intro1}.]
By Theorem~\ref{T:main1}, we have a cellular decomposition of $G/L$ 
that is given by the preimages under $\pi_{L,P}$ of the Borel orbits from $G/P$. 
Therefore, the Chow ring of $G/L$ is isomorphic to the cohomology 
ring of $G/L$, and furthermore, the closures of the affine cells 
form a generating set for $A^*(G/L)$. But their images generate
the isomorphic ring $A^*(G/P)$ which is known to be 
given by $S^{W_L}/(S_+)^W$. This finishes the proof.
\end{proof}

\section{Clans}\label{S:4}

We begin by recalling some of the background on clans.
We loosely follow the presentations of~\cite{Yamamoto} and~\cite{WyserThesis}.

\begin{Definition} 
Let $p$ and $q$ be two positive integers such that $p+q=n$. 
A \emph{ $(p,q)$-clan} is an ordered set of $n$ symbols 
$( c_1 \dots c_n)$ such that: 
\begin{enumerate}
\item Each symbol $c_i$ is either $+$, $-$ or a natural number.
\item If $c_i \in \N$ then there is a unique $c_j$, $i\neq j$ such that $c_i=c_j$.
\item The difference between the numbers of $+$ and $-$ symbols in the clan is equal to $p-q$. 
If $q>p$, then we have $q-p$ more minus signs than plus signs. 
\end{enumerate}
(Strictly speaking, such a sequence should be called a `$(p,q)$-preclan'
and the object that is defined by the next sentence should be called a $(p,q)$-clan.)
Clans are determined only up to equivalence based on the positions 
of matching pairs of natural numbers in the clan; 
the particular values of the natural numbers are not important.
For example, $(+ 1 2 1 2 -)$ and $(+ 1 7 1 7  -)$ 
are equivalent $(3,3)$-clans.
We denote the set of all $(p,q)$-clans by $\mc{C}(p,q)$.
\end{Definition}

Clans represent $\mbf{B}_n$-orbits in the flag variety $\mbf{Fl}_n$. 
We now describe how to obtain flags to represent the orbit corresponding to a given clan. 
Recall that a (full) flag $F_\bullet$ is a sequence of vector subspaces $(V_i)_{i=0}^n$ such that
\[ 
\{0\}= V_0 \subset V_1 \subset V_2 \subset\dots \subset V_n = \C^n 
\] 
and $\dim V_i=i$ for all $i$. 
We find it convenient to write 
\[
F_\bullet = \langle v_1, v_2, \dots, v_n  \rangle 
\] 
to indicate that $F_\bullet$ is the flag with 
$V_i= \mt{span} \{ v_1, \dots, v_i \}$ for all $1\leq i \leq n $. 
Also, let us fix a (standard) basis $\{e_i\}_{i=1}^n$ for $\C^n$.

\begin{Definition}
A \emph{signed $(p,q)$-clan} is a $(p,q)$-clan $\gamma= (c_1\cdots c_n)$ 
together with a choice of assignment of ``signature'' $+$ or $-$ to the 
individual members of each pair of matching natural numbers in $\gamma$, 
such that if one of the numbers in the pair is assigned $+$, then the other is assigned $-$. 
The \emph{default signed clan} associated to $\gamma$ is the one that gives 
$c_i$ a signature of $-$ whenever $c_i=c_j \in \N$ and $i< j$, and is denoted by $\tilde{\gamma}$.
\end{Definition}

We use subscripts to indicate signatures when writing down signed clans; 
for instance $( + 1_+ 2_- 1_- 2_+ - )$ is a signed $(3,3)$-clan, 
while $\tilde{\gamma} = (+ 1_- 2_- 1_+ 2_+ -)$ is the default signed clan of $\gamma=(+ 1 2 1 2 -)$.

\begin{Theorem}[see \cite{Yamamoto}, Theorem 2.2.14]\label{algom}
Given a clan $\gamma\in \mc{C}(p,q)$, assign signatures 
to the natural numbers appearing in $\gamma$ so as to obtain 
a signed $(p,q)$-clan. 
Now fix a permutation $\sigma \in S_n$ such that:
\begin{align*}
1 \leq \sigma(i) \leq p \textrm{\quad    if $c_i = +$ or $c_i\in \N$ with signature $+$}, \\
p+1 \leq \sigma(i) \leq n \textrm{ \quad   if $c_i= -$ or $c_i \in \N$ with signature $-$.}
\end{align*}
Next define a flag $F_\bullet =\langle v_1, v_2,\dots, v_n\rangle$ by taking
\begin{align*}
&v_i = e_{\sigma(i)} &&  \textrm{\quad if $c_i= \pm$,}  &\\
&v_i =e_{\sigma(i)} + e_{\sigma(j)} && 
\textrm{\quad if $c_i= c_j\in \N$ and $c_i$ has signature $+$,}&\\
&v_i =-e_{\sigma(i)} + e_{\sigma(j)} && 
\textrm{\quad if $c_i= c_j\in \N$ and $c_i$ has signature $-$}. & 
\end{align*}
The flag so obtained represents an $\mbf{L}_{p,q}$-orbit,
$\mbf{Q}_\gamma:= \mbf{L}_{p,q} \cdot F_\bullet$ in $\mbf{SL}_n/\mbf{B}_n$. 
Repeating this process for all $\gamma\in \mc{C}(p,q)$ 
gives the complete collection of distinct $(\mbf{L}_{p,q},\mbf{B}_n)$-double cosets
in $\mbf{SL}_n$.
\end{Theorem}

\begin{Example}
Let us consider $\gamma=(+ 1 2 1 2 -)$ again. 
We will use the default signed clan $\tilde{\gamma} = (+ 1_- 2_- 1_+ 2_+ -)$ for signature, 
and then we choose $\sigma=(25)(34)$ as our permutation (in cycle notation). 
Then the corresponding representing flag for $\mbf{Q}_\gamma$ is given by 
\[
\langle e_1, e_3-e_5, e_2- e_4, e_3+e_5, e_2+e_4, e_6 \rangle. 
\]
\end{Example}

\begin{Remark} 
Note that the possible choices of $\sigma$ for a 
given clan differ by an element of $S_p \x S_q$, which is the Weyl group of $\mbf{L}_{p,q}$. 
This is one way of understanding why all choices yield representatives of the same $\mbf{L}_{p,q}$-orbit.
\end{Remark}

\begin{Remark}\label{flagmatrix} 
We can convert between the flag and coset descriptions of points in $\mbf{Fl}_n $ as follows. 
Given $F_\bullet=\langle v_1,\dots,v_n\rangle$, 
then we take $\hat{g}$ as the matrix with $i$-th column equal to $v_i$. 
To get an element of $\mbf{SL}_n$, define $g=(\frac{1}{\det \hat g})^{\frac{1}{n}} \hat{g}$, 
for some choice of $n$-th root of $1/{\det \hat{g}}$, 
and thus we may associate the coset $g\mbf{B}_n$ in $\mbf{SL}_n/\mbf{B}_n$ to the flag 
$F_\bullet$ in $\mbf{Fl}_n$. 
Given the coset $g\mbf{B}_n$, by taking $V_i$ to be the span of the first $i$ columns of 
any coset representative, we likewise recover the flag $F_\bullet$.
\end{Remark}

We now proceed with our purpose of describing which 
$\mbf{L}_{p,q}$-orbits lie in the same $\mbf{P}_{p,q}$-orbit in $\mbf{Fl}_n$. 
In fact, by using symmetry, we will describe which $\mbf{B}_n$-orbits of $\mbf{SL}_n/\mbf{L}_{p,q}$ 
lie over which Schubert cells of $\mbf{SL}_n/\mbf{P}_{p,q}$ under the canonical projection.
Let us make an auxiliary definition. 

\begin{Definition} 
The \emph{base clan} of a $(p,q)$-clan $\gamma$ is the clan 
obtained by replacing signed natural numbers by their signature in 
the default signed clan $\tilde{\gamma}$.
\end{Definition}

Again, for example, the base clan of $( + 1 2 1 2- )$ is $( + - - + + -)$.  
We remark here that the $(p,q)$-clans which arise as base clans 
(those with only $+$'s and $-$'s as symbols) are the closed $\mbf{L}_{p,q}$-orbits.
In this notation, we have a specific but a critical result towards 
showing our second main theorem. 

\begin{Theorem}\label{T:clanorbit} 
Let $\mbf{Q}_\gamma$ and $\mbf{Q}_\tau$ be $\mbf{L}_{p,q}$-orbits in $\mbf{Fl}_n$ corresponding to 
$(p,q)$-clans $\gamma$ and $\tau$. 
Then $\mbf{Q}_\gamma$ and $\mbf{Q}_\tau$ lie in the same $\mbf{P}_{p,q}$-orbit of $\mbf{Fl}_n$ if and only if 
$\gamma$ and $\tau$ have the same base clan.
\end{Theorem}

\begin{proof} 
First we show the sufficiency; 
if $\gamma$ has base clan $\tau$, 
then the representing flags for each $\mbf{L}_{p,q}$-orbit 
$\mbf{Q}_\gamma$ and $\mbf{Q}_\tau$ lie in the same $\mbf{P}_{p,q}$-orbit. 
Then all clans with base clan $\tau$ will lie in the same $\mbf{P}_{p,q}$-orbit. 
More precisely, we will exhibit parabolic group elements,  
which iteratively transform the representing flag for $\tau$ 
into the representing flag for $\gamma$.

Let $\gamma=(c_1\ldots c_n)$ and $\tau=(t_1\ldots t_n)$, and let 
$F_\gamma=\langle v_1,\dots, v_n \rangle$ and $F_\tau=\langle u_1, \dots, u_n \rangle$ 
be the corresponding flags constructed by Theorem~\ref{algom}, 
using $\tilde{\gamma}$ as the signed clan for $\gamma$ 
($\tau$ is already ``signed'' since it consists only of $+$'s and $-$'s). 
Additionally, we may also use the same permutation $\sigma$ to construct both flags, 
as each clan has the same signature. Then, for each pair of numbers $c_i=c_j$ with $i< j$, 
we have
\[
(v_i, v_j) = (e_s-e_t, e_s+ e_t) \quad \textrm{and} \quad (u_i,u_j) = (e_t, e_s) 
\]
for some $1\leq s \leq p$ and $p+1 \leq t \leq n$.

Let $p^{s,t}$ denote the matrix of the linear transformation defined by 
\[ 
e_t \longmapsto e_s-e_t, \quad e_i\longmapsto e_i \quad\text{ for } i\neq t.
\]
Note that $p^{s,t}$ is an element of the parabolic subgroup for suitable choices 
of $s$ and $t$; this can be verified by examining the block-matrix description of $\mbf{P}_{p,q}$.
Note also that the pairs of vectors
\[ 
(e_s-e_t, e_s+ e_t) \quad \textrm{and} \quad (e_s-e_t, e_s) 
\]
generate the same subspace, so it does not matter which one of the vectors 
$e_s+e_t$ or $e_s$ appears as $v_j$ in the flag $F_\gamma$.  
Hence, after we act on the flag $F_\tau$ 
by the appropriate element of the form $p^{s,t}$, 
one for each pair of natural numbers $c_i=c_j$ in $\gamma$, 
we obtain an equivalent presentation of $F_\gamma$. 
Thus $\mbf{Q}_\gamma$ is in the same $\mbf{P}_{p,q}$-orbit as $\mbf{Q}_\tau$. 

To show the converse, it suffices to show that orbits corresponding 
to distinct base clans $\tau$ and $\lambda$ lie in distinct $\mbf{P}_{p,q}$-orbits. 
Let $\tau=(t_1\cdots t_n)\neq\lambda=(l_1 \cdots l_n)$ and let 
$F_\tau=\langle u_1,\dots, u_n \rangle$ and $F_\lambda=\langle w_1,\dots, w_n \rangle$ 
be flags constructed to represent each orbit using Theorem~\ref{algom}. 
Now let $i$ be the least index such that $t_i \neq l_i$. 
Without loss of generality, we may assume that $t_i=-$ and $l_i=+$. 
Then we have $u_i=e_t$ for some $p+1\leq t \leq n$ and $w_i = e_s$ for some $1\leq s \leq p$. 
For these flags to be in the same $\mbf{P}_{p,q}$-orbit we would need to be able to carry $e_s$ 
to a vector with non-zero $e_t$ component via some $p\in \mbf{P}_{p,q}$. 
This would force a non-zero entry in the $(t,s)$-entry of the matrix of $p$, 
but since $p$ has a block-diagonal form with zero $(i,j)$-entry whenever $i>p$ and $j\leq p$, 
this is impossible. 
Thus, the $\mbf{L}_{p,q}$-orbits $\mbf{Q}_\tau$ and $\mbf{Q}_\lambda$ are in distinct $\mbf{P}_{p,q}$-orbits.

\end{proof}

\begin{Definition}\label{D:seconddef}
We call the collection of clans with base clan 
$\gamma$ the \emph{sect} of $\gamma$, denoted by $\Sigma_\gamma$. 
\end{Definition}

By abuse of terminology, we will also use this term (sect) to describe collections of 
any of the objects ($\mbf{B}_n$-orbits, $\mbf{L}_{p,q}$-orbits, double cosets, etc.) 
parameterized by $\Sigma_\gamma$. 
We will show below that the Definitions~\ref{D:seconddef}
and~\ref{D:firstdef} are concordant.
Let us also make another definition that will be useful in what follows in order to 
assign representing flags and matrices to clans in a standard way.

\begin{Definition} \label{defaults} 
Let $\gamma= (c_1\dots c_n)$ be a $(p,q)$-clan with default 
signed clan $\tilde{\gamma}$. 
Then let $S_{{\gamma}}=\{c_{s_1}, \dots, c_{s_k} \}$ 
denote the set of symbols such that $1\leq s_l \leq p$, 
the signature of $c_{s_l}$ in $\tilde{\gamma}$ is $-$, 
and ${s_l} < {s_{l+1}}$ for all $1\leq l \leq k-1$. 
Let $T_{{\gamma}}=\{c_{t_1},\dots, c_{t_k}\}$ 
be the set of symbols such that $p+1 \leq {t_l} \leq n$, 
the signature of $\tilde{\gamma}$ is $+$, and ${t_l} < {t_{l+1}}$ 
for all $1\leq l \leq k-1$. 
Now define the permutation $\sigma_{{\gamma}}$ by 
\[ 
\sigma_{{\gamma}}(i) = \begin{cases}
i \qquad \textrm{ if } c_i \not\in S_{{\gamma}} \cup T_{{\gamma}} \\
t_l \qquad \textrm{ if } i={s_l} \\
s_l \qquad \textrm{ if } i={t_l} .
\end{cases}
\]
We call $\sigma_\gamma$ the \emph{default permutation} 
associated to $\gamma$. 
We will call the flag obtained via Theorem~\ref{algom} by 
using the signatures of $\tilde{\gamma}$ and permutation 
$\sigma_\gamma$ the \emph{default flag} associated to $\gamma$. 
Finally, a matrix constructed as in Remark~\ref{flagmatrix} 
from the default flag of $\gamma$ will be called 
a \emph{default matrix} associated to $\gamma$; 
note that there is a choice of $n$-th root of unity involved in constructing such a matrix. 
\end{Definition}

From now on, we will reserve the symbols 
$F_\gamma$ and $g_\gamma$ respectively to mean the default 
flag and default matrix of a clan $\gamma$. 
To illustrate these definitions, we turn to our running example 
$\gamma=(+ 1 2 1 2  - )$. Then, $\sigma_\gamma=(2 4) (3 5)$, and
\[ 
F_\gamma= \langle e_1, e_2- e_4, e_3-e_5, e_2+e_4, e_3+e_5, e_6 \rangle. 
\]
Finally, 
\[
g_\gamma = (\frac{1}{4})^\frac{1}{6} \begin{pmatrix}
1 & 0 & 0 & 0 & 0 & 0 \\
0 & 1 & 0 & 1 & 0 & 0 \\
0 & 0 & 1 & 0 & 1 & 0 \\
0 &-1 & 0 & 1 & 0 & 0 \\
0 & 0 &-1 & 0 & 1 & 0 \\
0 & 0 & 0 & 0 & 0 & 1
\end{pmatrix}.
\]
\begin{Remark} \label{involution} If $\tau$ is a base clan, then $g_\tau$ is a scalar 
multiple of the matrix of a permutation which is an involution, 
where the multiple is either an $n$-th root or a $2n$-th root of unity, 
depending on whether the permutation is even or odd. 
This follows from that the determinant of a permutation matrix is the sign of the permutation, 
and that the reciprocal of a $k$-th root of unity is also a $k$-th root of unity.
\end{Remark}

Let us now make the identification of $\mbf{L}_{p,q}$-orbits in $\mbf{Fl}_n$ 
and $\mbf{B}_n$-orbits in $\mbf{SL}_n/\mbf{L}_{p,q}$ more explicit.  
First note that for any subgroups $H, K \subseteq G$, there is a natural 
bijection between the set of $(H,K)$ double cosets and the set of 
$(K,H)$ double cosets that is given by  
\begin{align*}
\varphi: HgK &\longrightarrow K g^{-1}H  \\
 hgk & \longmapsto kg^{-1}h \notag
\end{align*}
Thus, given a clan $\gamma$, we have the induced map 
\[
\mbf{Q}_\gamma =
\mbf{L}_{p,q} g_\gamma \mbf{B}_n/\mbf{B}_n \longrightarrow 
\mbf{B}_n g_\gamma^{-1}\mbf{L}_{p,q}/\mbf{L}_{p,q}=:R_\gamma 
\]
that takes the the $\mbf{L}_{p,q}$-orbit $\mbf{Q}_\gamma$ to a unique 
$\mbf{B}_n$-orbit in $\mbf{SL}_n/\mbf{L}_{p,q}$, 
which we denote by $R_\gamma$ since it depends only on the clan $\gamma$ via $g_\gamma$.

Now we return to Schubert cells, and $\mbf{Gr}(p, n)$. 
\begin{Remark} \label{planematrix} 
Similarly to Remark~\ref{flagmatrix}, we can move between points of $\mbf{Gr}(p, n)$ 
and cosets $g\mbf{P}_{p,q}$ as follows. Given $x\in \mbf{Gr}(p, n)$, 
as a vector subspace, $x$ is the span of linearly independent vectors $\{v_1,\dots ,v_p\}$. 
Let $W= V^\perp$, the orthogonal complement of $V$. 
Then $W$ is an $n-p=q$ dimensional subspace of $\C^n$ 
spanned by some linearly independent collection $\{v_{p+1},\dots ,v_n\}$. 
If we let $\tilde{g}$ be the matrix whose $i$-th column is $v_i$, 
and define $g=(\frac{1}{\det \tilde{g}})^\frac{1}{n}\tilde{g} \in \mbf{SL}_n$, 
then the coset $g\mbf{P}_{p,q}$ represents the $p$-plane $x$ in $\mbf{Gr}(p, n)$. 
We recover $x$ by taking the span of the first $p$ columns of any matrix in the coset $g\mbf{P}_{p,q}$.
\end{Remark}

Theorem~\ref{T:clanorbit} describes which $\mbf{L}_{p,q}$-orbits of $\mbf{Fl}_n$ 
lie in the same $\mbf{P}_{p,q}$-orbit, effectively giving a decomposition of the double cosets 
$\mbf{P}_{p,q}g_\tau \mbf{B}_n$, where $\tau$ is the default matrix associated to a base clan 
$\tau$ via Theorem~\ref{algom}. 
We are now ready to give a restatement of Theorem~\ref{T:clanorbit} in terms of the map 
$\pi_{\mbf{L}_{p,q},\mbf{P}_{p,q}} : \mbf{SL}_n/ \mbf{L}_{p,q} \to \mbf{SL}_n/\mbf{P}_{p,q}$.

\begin{Proposition}\label{projbase}
Let $R_\gamma$ and $R_\tau$ be $\mbf{B}_n$-orbits in 
$\mbf{SL}_n/\mbf{L}_{p,q}$ corresponding to $(p,q)$-clans 
$\gamma$ and $\tau$, and let 
\hbox{$\pi_{\mbf{L}_{p,q},\mbf{P}_{p,q}}: \mbf{SL}_n/\mbf{L}_{p,q} \to \mbf{SL}_n/\mbf{P}_{p,q}$} 
denote the canonical projection. 
Then $\pi_{\mbf{L}_{p,q},\mbf{P}_{p,q}}(R_\gamma)= \pi(R_\tau)$ if and only if 
$\gamma$ and $\tau$ have the same base clan (are in the same sect).
\end{Proposition}

\begin{proof}
In view of the preceding discussion, we can represent the $\mbf{B}_n$-orbits as
\[ 
R_\gamma= \mbf{B}_n g_\gamma^{-1} \mbf{L}_{p,q}/\mbf{L}_{p,q} \qquad 
\text{and} \qquad R_\tau
= \mbf{B}_n g_\tau^{-1} \mbf{L}_{p,q}/\mbf{L}_{p,q} 
\]
with projections 
\[ 
\pi_{\mbf{L}_{p,q},\mbf{P}_{p,q}}(R_\gamma)
= \mbf{B}_n g_\gamma^{-1} \mbf{P}_{p,q}/\mbf{P}_{p,q} \qquad 
\text{and} \qquad \pi_{\mbf{L}_{p,q},\mbf{P}_{p,q}}(R_\tau)
= \mbf{B}_n g_\tau^{-1} \mbf{P}_{p,q}/\mbf{P}_{p,q}  .
\]
Then $\pi_{\mbf{L}_{p,q},\mbf{P}_{p,q}}(R_\gamma)=\pi_{\mbf{L}_{p,q},\mbf{P}_{p,q}}(R_\tau)$ 
if and only if 
$\mbf{B}_n g_\gamma^{-1} \mbf{P}_{p,q}  = \mbf{B}_n g_\tau^{-1} \mbf{P}_{p,q}$ 
which holds if and only if 
$\mbf{P}_{p,q} g_\gamma  \mbf{B}_n  = \mbf{P}_{p,q} g_\tau \mbf{B}_n$. 
From Theorem~\ref{T:clanorbit}, we know this happens 
if and only if $\gamma$ and $\tau$ have the same base clan. 
\end{proof}

Since $\pi_{\mbf{L}_{p,q},\mbf{P}_{p,q}}(R_\gamma)$ is a $\mbf{B}_n$-orbit in 
$\mbf{SL}_n/\mbf{P}_{p,q}$, that is a Schubert cell, 
it remains only to describe which Schubert cell it is. 
As in~\cite[Section 1]{Brion-Lectures}, a Schubert cell $C_I$ is determined 
by a choice of $p$ standard basis vectors, \hbox{$I= \{e_{i_1},\dots, e_{i_p}\}$}. 
Thus, there are $\binom{n}{p}$ cells.  We can bijectively associate base clans to sets $I$ 
by defining $\gamma_I=(c_1 \dots c_n)$ by
\begin{equation}\label{baseclanI} c_i =\begin{cases}
+ \quad \text{ if } e_i \in I, \\
- \quad \text{ if } e_i \notin I .
\end{cases}\end{equation}
Note that under this assignment, after building the default flag, 
$F_{\gamma_I}=\langle v_1, \dots, v_n\rangle$, 
it is easy to check that $\{v_1,\dots,v_p\}=I$. 
So by Remark~\ref{planematrix}, $C_I = \mbf{B}_n g_{\gamma_I} \mbf{P}_{p,q} /\mbf{P}_{p,q}$. 
Finally, we have our promised result.

\begin{Theorem}\label{decomp} 
Let $C_I$ be a Schubert cell of $\mbf{SL}_n/\mbf{P}_{p,q}$, 
and let 
\hbox{$\pi_{\mbf{L}_{p,q},\mbf{P}_{p,q}}: \mbf{SL}_n/\mbf{L}_{p,q} \to \mbf{SL}_n/\mbf{P}_{p,q}$} 
be the canonical projection. 
We associate to $I$ a base clan $\gamma_I$ as in equation \ref{baseclanI}, 
and we denote the sect of $\gamma_I$ by $\Sigma_I$. 
If $R_\gamma$ denotes the $\mbf{B}_n$-orbit of 
$\mbf{SL}_n/\mbf{L}_{p,q}$ associated to $\gamma$, then
\begin{equation}
\pi_{\mbf{L}_{p,q},\mbf{P}_{p,q}}^{-1}(C_I) 
= \bigsqcup_{\gamma \in \Sigma_I} R_\gamma.
\end{equation}
\end{Theorem}

\begin{proof} 
The pre-image of $C_I$ decomposes as a disjoint union of $\mbf{B}_n$-orbits 
$R_\gamma$ as a consequence of the fact that $\pi_{\mbf{L}_{p,q},\mbf{P}_{p,q}}$ 
is $\mbf{SL}_n$-equivariant, 
so in particular $\pi$ is $\mbf{B}_n$-equivariant. It remains to determine for which 
$\gamma$ we have $\pi_{\mbf{L}_{p,q},\mbf{P}_{p,q}}(R_\gamma)=C_I$. 
Proposition~\ref{projbase} tells 
us that $\mbf{B}_n$-orbits $R_\gamma$ and $R_\tau$ project into the same Schubert cell 
if and only if they are members of the same sect. 
Thus we only have to prove that the $\mbf{B}_n$-orbit $R_{\gamma_I}$ 
projects to the Schubert cell $C_I$.

We know that $\pi_{\mbf{L}_{p,q},\mbf{P}_{p,q}}(R_{\gamma_I})
=\mbf{B}_n g_{\gamma_I}^{-1} \mbf{P}_{p,q} /\mbf{P}_{p,q}$. 
From remark \ref{involution}, we know that 
\[
g_{\gamma_I} = \zeta^{-1}   u 
\]
where $\zeta$ is an $n$-th or $2n$-th root of unity and $u$ is the permutation matrix of an involution. 
This implies that 
\[
g_{\gamma_I}^{-1}= \zeta u = \zeta^2 g_{\gamma_I}= \zeta^2 I_n \cdot g_{\gamma_I}.
\]
Since $\zeta^2$ is an $n$-th root of unity, $\zeta^2 I_n$  is in the maximal torus of 
diagonal elements $\mbf{T} \subseteq  \mbf{SL}_n$. 
Since $\mbf{T}$ is contained in $\mbf{B}_n$, 
we have
\[
\pi_{\mbf{L}_{p,q},\mbf{P}_{p,q}}(R_{\gamma_I})
= \mbf{B}_n g_{\gamma_I}^{-1} \mbf{P}_{p,q} /\mbf{P}_{p,q} 
= \mbf{B}_n (\zeta^2 I_n \cdot g_{\gamma_I}) \mbf{P}_{p,q} /\mbf{P}_{p,q} 
= \mbf{B}_n g_{\gamma_I} \mbf{P}_{p,q} /\mbf{P}_{p,q} 
= C_I. 
\]
\end{proof}

As a final remark, this theorem indicates a bijection between Schubert cells 
$C_I$ and sects $\Sigma_I$; every sect $\Sigma_I$ is nonempty because 
it contains the clan $\gamma_I$, and every base clan is of the form $\gamma_I$ 
for some collection of $p$ standard basis vectors $I$, 
as prescribed by the association of (\ref{baseclanI}). 
We may continue to refer to sects indexed either by a particular member $\gamma$, 
or by the subset $I$ associated to the common base clan $\gamma_I$.

\section{The Bruhat Order}\label{S:5}

The Bruhat order for type AIII symmetric space, that is our variety of polarizations,
was first fully described in \cite{Wyser16}.  
After describing the specifics of the Bruhat order for $G/L$, we illustrate an example, 
and examine the interaction of the ordering with the decomposition given by Theorem \ref{decomp}.

Let $X_I=\overline{C_I}=\overline{\mbf{B}_n g_{\gamma_I} \mbf{P}_{p,q}/\mbf{P}_{p,q}}$ 
be a Schubert variety in $\mbf{Gr}(p,n)=\mbf{SL}_{n}/\mbf{P}_{p,q}$. 
Define a partial order on Schubert cells by \hbox{$C_I \leq C_J$} if and only if 
$X_I \subseteq  X_J$. 
Let $I= \{e_{i_1}, \dots, e_{i_p} \}$ and $J= \{e_{j_1}, \dots, e_{j_p}\}$ be ordered 
so that $i_k \leq i_{k+1}$ and $j_k\leq j_{k+1}$ for all $k$ from $1$ to $p-1$. 
Then we say that {${I \leq J}$ if and only if $i_k \leq j_k$ for all $k$ from $1$ to $p$}. 
From this, one has 
\[
X_J= \bigcup_{I\leq J} C_I ,
\]
so $C_I \leq C_J$ in the Bruhat order if and only if $I\leq J$ in the order described above,
see~\cite[Section 1.1.4]{Brion-Lectures}.

As Schubert cells in a Grassmannian are parametrized by lattice paths, 
the Bruhat order also translates to a nice description in lattice paths as follows. 
The collection of sets of the form $I=\{e_{i_1},\dots, e_{i_p}\}$ are in bijection 
with the paths from the origin to the point $(p,q)$ possible by 
taking only ``north" and ``east" steps ($N$ and $E$) as follows. 
Starting at the origin, move one unit east if $e_1\in I$, 
move one unit north if not. From this new point point, move one unit east if $e_2\in I$, 
move one unit north if not, etc. After $n$ steps, such a path necessarily terminates at the point $(p,q)$.

Having drawn paths for $I$ and $J$ on the same grid, 
the condition $I \leq J$ is equivalent to the path of $I$ 
lying weakly beneath that of $J$.
Furthermore, the dimension of $C_I$ is equal to the number of 
boxes beneath the path (shaded in Figures~\ref{fig:p2q1} and~\ref{fig:p2q2}). 
There is a unique closed Schubert cell which is zero dimensional.  
The Bruhat order on Schubert cells is drawn in lattice path form at the bottom of 
Figures~\ref{fig:p2q1} and~\ref{fig:p2q2} for some small examples.

 One has no trouble recovering the $p$-plane which 
 represents the corresponding Borel orbit, by taking the span 
 of the standard basis vectors indexed by the positions of the 
 eastward steps in the path sequence. 
 We also remark here that the path sequence of $I$ is evident in $\gamma_I$ 
 as defined by~(\ref{baseclanI}); north and east steps correspond to $-$ and $+$, respectively.

Let $G$ be an algebraic group and let $H$ be a closed subgroup in $G$. 
Let $B$ be a Borel subgroup in $G$. 
By analogy, one defines a partial order on Borel orbits in 
$G/H$ given by the containment of their closures. 
This also induces a partial on all other corresponding families of objects.
In our situation we have partial orders on the following sets of objects: 
\begin{itemize}
\item double cosets of $\mbf{B}_n \backslash \mbf{SL}_n / \mbf{L}_{p,q}$.
\item double cosets of $\mbf{L}_{p,q} \backslash \mbf{SL}_n / \mbf{B}_n$. 
\item $\mbf{L}_{p,q}$-orbits $\mbf{Q}_\gamma$ in $\mbf{Fl}_n$.
\item $(p,q)$-clans $\gamma$ in $\mc{C}(p,q)$.
\item signed $(p,q)$-involutions $\pi$ in $I_{p,q}^\pm$ (see \cite{CanUgurlu1}). 

\end{itemize}

We will refer to all of these (isomorphic) partial orders as 
``the Bruhat order,'' exchanging freely between the various descriptions of the underlying sets. 

Wyser gives the following concise description of the Bruhat 
order on clans corresponding to $\mbf{GL}_p\times \mbf{GL}_q$ 
orbits in the full flag variety of $\mbf{GL}_n$ in \cite{Wyser16}. 
Since the flag varieties of $\mbf{SL}_n$ and $\mbf{GL}_n$ are the same, 
and a matrix of $\mbf{GL}_p\times \mbf{GL}_q$ is a matrix of 
$S(\mbf{GL}_p\times \mbf{GL}_q)$ times a scalar multiple of the 
identity matrix (which fixes a flag), the orbits and their structure are the same.

Let $\gamma$ be a $(p,q)$-clan that is given by $\gamma=(c_1 \ldots c_n)$. 
For $i\in \{1,\dots, n\}$, let $\gamma ( i ; +)$ denote the number of plus signs 
and pairs of equal natural numbers among the symbols $c_1 \ldots c_i$. 
Similarly, let $\gamma(i ; -)$ denote the number of minus 
signs and pairs of equal natural numbers among the symbols $c_1 \ldots c_i$. 
Finally let, $\gamma(i, j)$ be the number of pairs of natural numbers 
$c_s=c_t$ with $s\leq i < j < t$.  To illustrate this on an example, 
we consider $\gamma=(+1212-)$. Then, as $i$ varies in $\{1,\dots,6\}$ we 
obtain the sequence $\gamma(i; +)= 1, 1 ,1 , 2, 3, 3 $, and $\gamma(3,4) = 1$.

\begin{Theorem}[\cite{Wyser16}, Theorem 1.2]\label{T:Wyser}
Let $ \gamma$, $\tau$ be $(p,q)$-clans. 
Then $\gamma \leq \tau$ if and only if the three following inequalities hold for all $i$, $j$. 
\begin{enumerate}
\item $\gamma(i ; +) \geq \tau ( i ; +)$;
\item $\gamma(i ; - ) \geq \tau (i ; - )$;
\item $\gamma (i ; j) \leq \tau( i ; j) $.
\end{enumerate}
\end{Theorem}

To illustrate this order on clans, we draw order 
diagrams for the cases $p=2$, $q=1$, and $p=q=2$, 
giving the same color to all clans whose $\mbf{B}_n$-orbits in $\mbf{SL}_n / \mbf{L}_{p,q}$ 
lie over the same Schubert cell, according to Theorem~\ref{decomp}.
Note that the Bruhat order on clans ($\mbf{B}_n$-orbits in $\mbf{SL}_n / \mbf{L}_{p,q}$) reflects the 
Bruhat order on the Schubert cells ($\mbf{B}_n$-orbits in $\mbf{SL}_n / \mbf{P}_{p,q}$) down below 
in the sense that if that if $\pi(R_\gamma) \leq \pi(R_\tau)$ 
then $\gamma \leq \omega$  for some $\omega$ in the same sect as $\tau$.

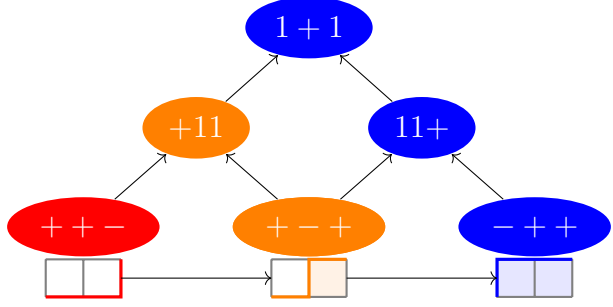
\begin{figure}[h] \caption{Inclusion order on
Borel orbit closures in $\mbf{SL}_3 / \mbf{L}_{2,1}$.}\label{fig:p2q1} 
\centering{\begin{tikzpicture}[
    big/.style={ellipse, draw=none, fill=blue,
        text centered, anchor=north, text=white},
    mid/.style={ellipse, draw=none, fill=orange, 
        text centered, anchor=north, text=white},
    small/.style={ellipse, draw=none, fill=red, 
        text centered, anchor=north, text=white},
    level distance=0.5cm, growth parent anchor=south
]
\node (big0) [big] {$1+1$} [<-] [sibling distance=3cm]
        child{ 
            node (mid0) [mid] {$+11$}
            child{
                node (small0) [small] {$++-$}
                }
                child{
                	node (mid1) [mid] {$+-+$}
		}
	      }
	 child{ 	
	 	node (big1) [big]  {11+} 
		child{ node (mid1) [mid] {$+-+$}
			}
		child{ node (big2) [big] {$-++$}
			}
		}	   ;
\fill[orange!10] (0, -4) -- (0,-3.5) -- (0.5, -3.5) -- (0.5, -4);
\fill[blue!10] (2.5, -4) -- (2.5,-3.5) -- (3.5, -3.5) -- (3.5, -4);

\draw[step=.5cm,gray,thick] 
(-3.5,-4) grid (-2.5,-3.5) ;
\draw[step=.5cm,gray,thick] 
(-0.5,-4) grid (0.5,-3.5) ;
\draw[step=.5cm,gray,thick] 
(2.49,-4) grid (3.5,-3.5) ;

\draw[color=red, very thick] (-3.5,-4) -- (-2.5, -4) -- (-2.5,-3.5) ;
\draw[color=orange, very thick] (-0.5,-4) -- (0, -4) -- (0,-3.5) -- (0.5, -3.5)  ;
\draw[color=blue, very thick] (2.5,-4) -- (2.5, -3.5) -- (3.5,-3.5) ;

\draw[->] (-2.5,-3.75) -- (-0.5, -3.75); 
\draw[->] (0.5,-3.75) -- (2.5, -3.75);

\node at (0,-4.5) {Bruhat order on Schubert varieties in $\mbf{Gr}(2, 3)$.};
\end{tikzpicture}}
\end{figure}

\begin{figure}[h] 
\centering
\caption{Inclusion order on
Borel orbit closures in $\mbf{SL}_4/\mbf{L}_{2,2}$.}\label{fig:p2q2}
\begin{tikzpicture}[
    teal/.style={ellipse, minimum width=2cm, draw=none, fill=teal,
        text centered, anchor=north, text=white},
    purp/.style={ellipse, minimum width=2cm, draw=none, fill=violet, 
        text centered, anchor=north, text=white},
    blue/.style={ellipse, minimum width=2cm, draw=none, fill=blue,
        text centered, anchor=north, text=white},
    mage/.style={ellipse, minimum width=2cm, draw=none, fill=magenta, 
        text centered, anchor=north, text=white},
    oran/.style={ellipse, minimum width=2cm, draw=none, fill=orange, 
        text centered, anchor=north, text=white},
    red/.style={ellipse, minimum width=2cm, draw=none, fill=red, 
        text centered, anchor=north, text=white},
    level distance=0.5cm, growth parent anchor=south
]
\node (teal0) at (0,10) [teal] {$1221$}  ;
\node (purp0) at (-3,8) [purp] {$1+-1$}  ;
\node (teal1) at (0,8) [teal] {$1212$}  ;
\node (teal2) at (3,8) [teal] {$1-+1$}  ;
\node (mage0) at (-5,6) [mage] {$+1-1$}  ;
\node (blue0) at (-2.5,6) [blue] {$1+1-$}  ;
\node (purp1) at (0,6) [purp] {$1122$}  ;
\node (teal3) at (2.5,6) [teal] {$1-1+$}  ;
\node (teal4) at (5,6) [teal] {$-1+1$}  ;
\node (oran0) at (-6.25,4) [oran] {$+11-$}  ;
\node (mage1) at (-3.75,4) [mage] {$+-11$}  ;
\node (blue1) at (-1.25,4) [blue] {$11+-$}  ;
\node (purp2) at (1.25,4) [purp] {$11-+$}  ;
\node (purp3) at (3.75,4) [purp] {$-+11$}  ;
\node (teal5) at (6.25,4) [teal] {$-11+$}  ;
\node (red0) at (-6.25,2) [red] {$++--$}  ;
\node (oran1) at (-3.75,2) [oran] {$+-+-$}  ;
\node (mage2) at (-1.25,2) [mage] {$+--+$}  ;
\node (blue2) at (1.25,2) [blue] {$-++-$}  ;
\node (purp4) at (3.75,2) [purp] {$-+-+$}  ;
\node (teal6) at (6.25,2) [teal] {$--++$}  ;

\draw [->] (purp0) edge (teal0) (teal1) edge (teal0) (teal2) edge (teal0)

(mage0) edge (purp0) (blue0) edge (purp0) (purp1) edge (purp0)
(mage0) edge (teal1) (blue0) edge (teal1) (purp1) edge (teal1) (teal3) edge (teal1) (teal4) edge (teal1)
(purp1) edge (teal2) (teal3) edge (teal2) (teal4) edge (teal2)

(oran0) edge (mage0) (mage1) edge (mage0) 
(oran0) edge (blue0) (blue1) edge (blue0)
(mage1) edge (purp1) (blue1) edge (purp1) (purp2) edge (purp1) (purp3) edge (purp1)
(purp2) edge (teal3) (teal5) edge (teal3) 
(purp3) edge (teal4) (teal5) edge (teal4) 

(red0) edge (oran0) (oran1) edge (oran0)
(oran1) edge (mage1) (mage2) edge (mage1)   
(oran1) edge (blue1) (blue2) edge (blue1)  
(mage2) edge (purp2) (purp4) edge (purp2)
(blue2) edge (purp3) (purp4) edge (purp3)
(purp4) edge (teal5) (teal6) edge (teal5)    
;

\fill[orange!10, xshift=8pt, yshift=-10pt] (-4, 0) -- (-4,0.5) -- (-3.5, 0.5) -- (-3.5, 0);
\fill[magenta!10, xshift=8pt, yshift=5pt] (-1.5, 0) -- (-1.5,1) -- (-1, 1) -- (-1, 0);
\fill[blue!10, xshift=8pt, yshift=-25pt] (0.5, 0.5) rectangle (1.5,0) ;
\fill[violet!10, xshift=8pt, yshift=-10pt] (3, 0) -- (3,0.5) -- (3.5, 0.5) -- (3.5, 1) -- (4,1) -- (4,0);
\fill[teal!10, xshift=8pt, yshift=-10pt] (5.5, 1) rectangle (6.5, 0 ) ;
\draw[xshift=8pt, yshift=-10pt, step=.5cm,gray,thick] 
(-7,0) grid (-6,1) ;
\draw[xshift=8pt, yshift=-10pt, step=.5cm,gray,thick] 
(-4.5,0) grid (-3.5,1) ;
\draw[xshift=8pt, yshift=5pt, step=.5cm,gray,thick] 
(-2,0) grid (-1,1) ;
\draw[xshift=8pt, yshift=-25pt, step=.5cm,gray,thick] 
(0.5,0) grid (1.5,1) ;
\draw[xshift=8pt, yshift=-10pt, step=.5cm,gray,thick] 
(2.99,0) grid (4,1) ;
\draw[xshift=8pt, yshift=-10pt, step=.5cm,gray,thick] 
(5.49,0) grid (6.5,1) ;
\draw[color=red, xshift=8pt, yshift=-10pt, very thick] (-7,0) -- (-6, 0) -- (-6,1) ;
\draw[color=orange, xshift=8pt, yshift=-10pt, very thick] (-4.5,0) -- (-4, 0) -- (-4,0.5) -- (-3.5, 0.5) -- (-3.5,1) ;
\draw[color=magenta, xshift=8pt, yshift=5pt, very thick] (-2,0) -- (-1.5, 0) -- (-1.5,1) -- ( -1,1) ;
\draw[color=blue, xshift=8pt, yshift=-25pt, very thick] (0.5,0) -- (0.5, 0.5) -- (1.5,0.5) -- (1.5,1) ;
\draw[color=violet, xshift=8pt, yshift=-10pt, very thick] (3,0) -- (3, 0.5) -- (3.5,0.5) -- (3.5, 1) -- (4,1) ;
\draw[color=teal, xshift=8pt, yshift=-10pt, very thick] (5.5,0) -- (5.5,1) -- (6.5,1) ;

\draw[->, xshift=8pt, yshift=-10pt] (-6,0.5) -- (-4.5, 0.5); 
\draw[->, xshift=8pt, yshift=-10pt] (-3.5,0.5) -- (-2, 1); 
\draw[->, xshift=8pt, yshift=5pt] (-1,0.5) -- (3, 0.1); 
\draw[->, xshift=8pt, yshift=-10pt] (-3.5,0.5) -- (0.5, -0.1); 
\draw[->, xshift=8pt, yshift=-25pt] (1.5,0.5) -- (3, 1); 
\draw[->, xshift=8pt, yshift=-10pt] (4,0.5) -- (5.5, 0.5); 
\node at (0,-1.5) {Bruhat order on Schubert cells in $\mbf{Gr}(2, 4)$.};
\end{tikzpicture}
\end{figure}
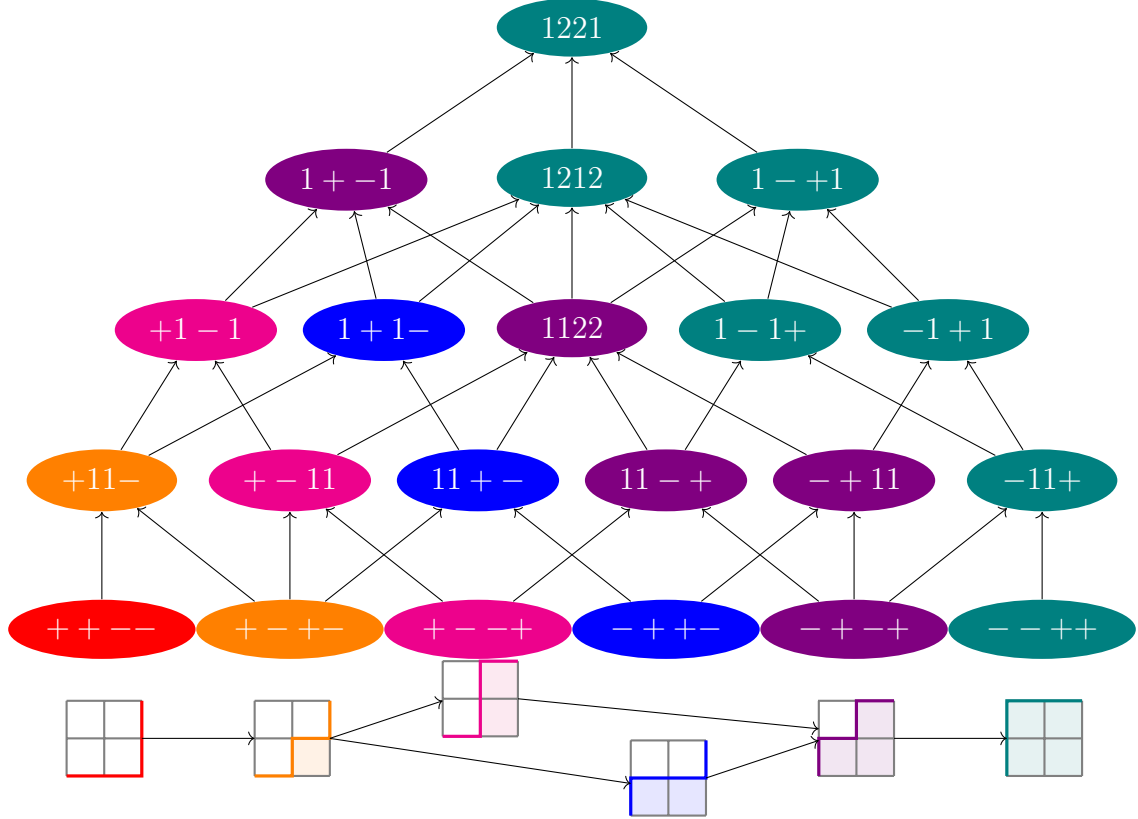

Note also that if $p\leq p'$ and $q\leq q'$ so that 
$p+q=n\leq n' =p'+q'$, then the Bruhat order on $(p,q)$-clans 
embeds into the order on $(p',q')$-clans by the map that takes 
\[
\gamma=(c_1\cdots c_n)\longmapsto\gamma'
=(\overbrace{+ \cdots +}^{p'-p}  c_1 \cdots c_n \overbrace{- \cdots -}^{q'-q} ).
\]
This embedding preserves sects and their internal order relations, 
as indicated by the red, orange and blue clans in Figure~\ref{fig:p2q2}.

As a final topic for this section, we remark upon the 
bijection of $(p,q)$-clans and weighted Delannoy paths described in \cite{CanUgurlu1} 
in the equivalent terminology of signed $(p,q)$-involutions.
As the sect $\Sigma_I$ corresponds to the lattice path obtained from $I$, 
one expects that the weighted Delannoy paths associated to all members 
of the sect will have features in common. 
One also hopes that the Bruhat order on clans can be described 
in terms of paths as it can for Schubert cells. 
This is true, but subject to some qualifications.

\begin{Remark} 
We list some properties relating the weighted Delannoy 
path of a clan to the Bruhat order and the sect decomposition. 
  
\begin{enumerate}

\item The weighted Delannoy path associated to $\gamma_I$ 
is the same as the lattice path of $I$. 

\item If $\gamma \leq \tau$ with $\tau\in \Sigma_\gamma$, 
then the weighted Delannoy path associated to $\tau$ lies weakly 
beneath that of $\tau$. Further, if these paths have weighted edges in common, 
then the weight on the edge for $\tau$ will be less than or equal 
to the weight on the edge for $\gamma$.

\item Clans from distinct sects may give the same 
underlying~\emph{unweighted} Delannoy path.

\end{enumerate}
\end{Remark}

See Figure 6.2 of \cite{CanUgurlu1} for illustration of the above properties. 
One can, of course, given the weighted Delannoy path of a clan 
$\gamma \in \Sigma_I$, reconstruct the lattice path of $I$. 
This is achieved by chasing through the associations

\[
\text{weighted Delannoy path} \leftrightarrow \text{signed } (p,q)\text{-involution} 
\leftrightarrow (p,q)\text{-clan} \rightarrow \text{base clan} \leftrightarrow \text{lattice path}.
 \]

We summarize this process in terms of path manipulations 
(as given by strings of moves $N$ or $E$) in the following algorithm. 
Note that the string of a weighted Delannoy path may additionally 
contain symbols of the form $(D, w)$, where $D$ indicates a diagonal step, 
and $w$ indicates the weight of the step.
\\

\textbf{Algorithm:}
\\

Let $W= K_1\dots K_r$ be a weighted $(p,q)$ Delannoy path (\cite{CanUgurlu1}, 
Definition 1.16) corresponding to a clan $\gamma\in\Sigma_I$. 
Construct a lattice path $L=M_1\cdots M_n$ from the origin to the point 
$(p,q)$ by the following process. 
The output $L$ will be the lattice path of the set $I$.
{\sffamily
\begin{enumerate}
\item \ [initialize]\\
 $L = \{ \} $.
\item \ [loop]\\
 for $(1\leq i \leq R)$:\\
\quad if $K_i = N$: \\
\quad \quad append $N$ to $L$ \\
\quad if $K_i = E$: \\
\quad \quad append $E$ to $L$\\
\quad else if $K_i=(D, w)$:\\
\quad \quad insert step $N$ at position $w$ \\
\quad \quad append $E$ to $L$\\
return $L$
\end{enumerate}
}
 

For our next result, without loss of generality we assume that $p\geq q$,
and we set
\begin{align}\label{A:r}
r:= \frac{ n-(p-q) }{2}.
\end{align}
In addition, we will denote by $\gamma_{I_0}$ the base clan of the big sect $\mt{Dense}(p,q)$,
so that, $I_0$ corresponds to the lattice path with $q$ north steps followed by $p$ east steps. 
\begin{Proposition}\label{P:orderideal}
For all positive integers $p,q$, and $r$ as in the previous paragraph, the following 
statements hold.
\begin{enumerate}
\item The big sect $\mt{Dense}(p,q)$ is a maximal upper order ideal in $\mc{C}(p,q)$. 
\item The unique minimal
element of $\mt{Dense}(p,q)$ is given by $\gamma_{I_0}$.
\item The unique maximal element of $\mt{Dense}(p,q)$ is given by 
$$
\gamma_{max}:= (12\ldots (r-1) r + + \ldots + + r (r-1)\ldots 21).
$$
\end{enumerate}
\end{Proposition}
\begin{proof}
We already know from Proposition~\ref{C:minmax1} that 
each sect has a unique maximal and a unique minimal element. 
Since base clans are the minimal elements of the Bruhat order,
$\gamma_{I_0}$ is the minimal element of $\mt{Dense}(p,q)$.
To see that $\gamma_{max}$ is the maximal element of 
$\mt{Dense}(p,q)$, hence, of $\mc{C}(p,q)$ we refer to~\cite[Section 7.2]{CanUgurlu1},
where the maximal element of $\mc{C}(p,q)$ is computed by using the 
signed $(p,q)$-involution notation, which is not hard to translate to clan
terminology. 

Next, we prove our claim about the order ideal property. 
Let $\gamma$ be a clan from $\mt{Dense}(p,q)= \Sigma_{I_0}$, and let $\tau$ 
be a clan from the sect $\Sigma_J$, so that $J\lneq I_0$. 
Towards a contradiction we assume that $\tau \geq \gamma$.
Then the closure of the Borel orbit corresponding to $\tau$ contains the 
Borel orbit corresponding to $\gamma$. 
Since the projection $\pi_{\mbf{L}_{p,q},\mbf{P}_{p,q}}$ is $\mbf{B}_n$-equivariant, 
this would imply that the closure of the Borel orbit of $\tau$ contains 
the dense $\mbf{B}_n$-orbit $C_{I_0}$ of $\mbf{Gr}(p,n)$. But this is not possible 
since the sect of $\tau$ is not equal to $\Sigma_{I_0}$. This contradiction shows
that $\mt{Dense}(p,q)= \Sigma_{I_0}$ is an upper order ideal in $\mc{C}(p,q)$. 
In particular, since it contains a minimal element and the unique maximal element of $\mc{C}(p,q)$,
$\mt{Dense}(p,q)$ is a maximal upper order ideal.
This finishes the proof. 
\end{proof}

\begin{Remark}
There is an algorithm for finding the {\em leader},
that is the maximal element, of a sect in $\mc{C}(p,q)$. 
We are not including it here for the brevity of our paper.
\end{Remark}

\section{Proof of Theorem~\ref{T:main2}}\label{S:6}

The Theorem~\ref{T:main2} states that 
the big sect $\mt{Dense}(p,q)$ is a maximal upper order ideal
in the Bruhat order on clans $\mc{C}(p,q)$, and if $p=q$,
then $\mt{Dense}(p,q)$ is isomorphic to the rook monoid $R_p$. 
We already proved the first claim in Proposition~\ref{P:orderideal}, so,
we will proceed with the assumption that $p=q$, hence $n=2p$. 

First, we construct a bijection between 
$\mt{Dense}(p,q)$ and the rook monoid $R_p$. 
Let $\gamma=(c_1\ldots c_n)$ be a $(p,p)$-clan. 
We will denote by $x_\gamma =(x_{i,j})_{i,j=1}^n$ the rook matrix 
that corresponds to $\gamma$. 
Since $\gamma$ is from $\mt{Dense}(p,q)$ we know that the base 
clan of $\gamma$ has $-$ signs in the first $p$ entries, and it has $+$ 
signs in the second $p$ entries. In particular, if $c_i$ is a 
number, where $1\leq i \leq p$, then the same number appears
as $c_j$ in $\gamma$ for some $j$ such that $p+1 \leq j \leq n$. 
Thus we define 
\begin{align}\label{A:themap}
\gamma \longmapsto x_\gamma := (x_{r,s})_{r,s=1}^p, \ \text{ where } x_{r,s} := 
\begin{cases}
1 & \text{ if $c_{r+p}$ is a number and $c_{r+p}= c_{s}$;}\\
0 & \text{ otherwise.} 
\end{cases}
\end{align}
It is clear from the construction of $x_\gamma$ that the map $\gamma \mapsto x_\gamma$ 
is a bijection. Furthermore, $x_\gamma \in R_p$. 
\begin{Remark}\label{R:fromrook}
For each rook matrix $y$ in $R_p$, there exists a unique involution $\tilde{y}$ in $S_n$
that is obtained by placing $y$ in the lower-left $p\x p$ block, placing its transpose 
$y^\top$ to the upper-right $p\x p$ block, and by adding 1's along the permissible 
diagonal entries. 
\end{Remark}

Before showing that the map $\gamma \mapsto x_\gamma$ is a poset 
isomorphism between $\mt{Dense}(p,q)$ and $R_p$, we 
will show that $x_\gamma$ is a part of an involution in $S_n$. 
To this end, let us denote by $\tilde{\sigma}_\gamma$ the underlying 
signed $(p,p)$-involution of $\gamma$, and let us denote by 
$\sigma_\gamma$ the involution that is obtained from $\tilde{\sigma}_\gamma$ 
by removing the signs of its fixed points. We call $\sigma_\gamma$
the {\em underlying involution of $\gamma$}.
We claim that $x_\gamma$ is actually the lower-left $p\x p$ submatrix 
of $\sigma_\gamma$. 
Indeed, the 2-cycles in $\tilde{\sigma}_\gamma$ are given by 
the transpositions $(i,j)$, where $c_i = c_j$ and $c_i$ is a number,
see~\cite[Lemmas 2.3 and 2.5]{CanUgurlu2}.
Furthermore, since $\gamma \in \mt{Dense}(p,q)$, we know that 
$1\leq i \leq p$ and $p+1 \leq j \leq n$. 
When we represent $\sigma_\gamma $ by an $n\x n$ permutation matrix, 
the transposition $(i,j)$ introduces a 1 at the $(i,j)$-th 
position as well as a 1 at the $(j,i)$-th position. 
This shows that the lower-left corner of the matrix $M_\gamma$ of $\sigma_\gamma $ 
is exactly the same as the rook matrix $x_c$ that we defined in (\ref{A:themap}).
We will illustrate this by an example. 
\begin{Example}
Let $\gamma=(-1212+)$ be an element from $\mt{Dense}(3,3)$. 
Then we see that the cycle decomposition of 
the underlying involution of $\gamma$ is given by 
$\sigma_\gamma = (2,4)(3,5)(1)(6)$. In one-line notation
this permutation is equal to $145236$, and its matrix is given 
by 
\begin{align}\label{A:thematrixof1}
M_{(-1212+)}= 
\begin{bmatrix}
1 & 0 & 0 & 0 & 0 & 0 \\
0 & 0 & 0 & 1 & 0 & 0 \\
0 & 0 & 0 & 0 & 1 & 0 \\
0 & 1 & 0 & 0 & 0 & 0 \\
0 & 0 & 1 & 0 & 0 & 0 \\
0 & 0 & 0 & 0 & 0 & 1 
\end{bmatrix}
\end{align}
On the other hand, according to (\ref{A:themap}), 
the rook matrix $x_\gamma$ associated with $\gamma=(-1212+)$ is 
\begin{align}\label{A:thematrixof2}
x_{(-1212+)}=
\begin{bmatrix}
0 & 1 & 0 \\
0 & 0 & 1 \\
0 & 0 & 0 
\end{bmatrix}
\end{align}
Clearly the matrix in (\ref{A:thematrixof2}) appears in the lower-left $3\x 3$ submatrix
of the $6\x 6$ matrix in (\ref{A:thematrixof1}).

\end{Example}

Next, we proceed to prove that the map $\gamma \mapsto x_\gamma$ is order preserving. 
A corollary of Theorem~\ref{T:Wyser}, see~\cite[Corollary 2.7]{Wyser16}, states that  
if $ \gamma$, $\tau$ are two $(p,q)$-clans, then 
$\gamma \leq \tau$ if and only if the three following inequalities hold:
\begin{enumerate}
\item $\gamma(i ; +) \geq \tau ( i ; +)$ and $\gamma(i ; - ) \geq \tau (i ; - )$ for all $i\in \{1,\dots, n\}$; 
\item $\sigma_\gamma \leq \sigma_\tau$ in the Bruhat order on involutions. 
\end{enumerate}

For every $(p,p)$-clan in $\mt{Dense}(p,q)$, $-$'s can only occur among the first $p$ 
symbols, and $+$'s can only occur among the last $p$ symbols. 
Therefore, as $\gamma$ runs in the set $\mt{Dense}(p,q)$ the change 
in the statistics $\gamma(i ; +)$ and $\gamma(i ; - )$ depend only on 
how many natural numbers appear in $(c_1\ldots c_i)$. 
Passing to the matrix $M_\gamma$ of the underlying involution, these
statistics essentially translate to rank conditions on the suitable 
lower-left submatrices of $M_\gamma$.
But since the Bruhat order on involutions is also given by the rank-control matrices,
see~\cite{BagnoCherniavsky}, we see that 
the Bruhat order on the underlying involutions of $\gamma \in \mt{Dense}(p,q)$ 
is enough to determine the Bruhat order on the sect $\mt{Dense}(p,q)$.
In other words, we have 
\begin{align*}
\gamma \leq \tau \iff \sigma_\gamma \leq \sigma_\tau \ \text{ for $\gamma,\tau \in \mt{Dense}(p,q)$}.
\end{align*}
Now, since $x_\gamma$ and $x_\tau$ are the lower-left $p\x p$ submatrices
of $\sigma_\gamma$ and $\sigma_\tau$, respectively, we see that 
\begin{align}\label{A:implication}
\sigma_\gamma \leq \sigma_\tau \implies x_\gamma \leq x_\tau \ \text{ for $\gamma,\tau \in \mt{Dense}(p,q)$}.
\end{align}
We claim that the converse of implication (\ref{A:implication}) is also true.
Indeed, let us assume that $x_\gamma \leq x_\tau$ holds for $\gamma,\tau \in \mt{Dense}(p,q)$.
Clearly, for any index $j$, the number of 0-rows of $x_\gamma$ after its $j$-th row 
must be at least the number of 0-rows of $x_\tau$ after its $j$-th row. 
The same statement holds for the number of 0-columns before the $j$-th
column in each. 
Although this implies that 
the nonzero entries along the diagonal of the unique involution matrix $M_\gamma$ come (weakly) before 
those of $M_\tau$, this does not change the order between
the rank-control matrices, hence, we see that 
$M_\gamma \leq M_\tau$. In particular, this implies that $\sigma_\gamma \leq \sigma_\tau$.
This finishes the proof of our second main result.

\bibliography{ReferencesMC}
\bibliographystyle{plain}

\end{document}